\documentclass[11pt]{article}
\usepackage{amsmath, amssymb, amscd, amsthm, amsfonts,bbm}
\usepackage{mathrsfs}
\usepackage{graphicx}
\usepackage{float} 
\usepackage{xcolor}
\usepackage{subfigure}
\usepackage{hyperref}
\usepackage{enumerate}
\numberwithin{equation}{section}
\usepackage{listings}
\usepackage{setspace}
\title{Boundary four-point connectivities of conformal loop ensembles}
\author{Gefei Cai\footnote{\url{caigefei1917@pku.edu.cn}, Peking University.}}
\date{\today}

\newcommand{\R}{\mathbbm{R}}

\newcommand{\C}{\mathbbm{C}}
\newcommand{\Z}{\mathbbm{Z}}
\newcommand{\T}{\mathbbm{T}}

\newcommand{\cP}{\mathcal{P}}
\newcommand{\cL}{\mathcal{L}}
\newcommand{\sP}{\mathscr{P}}
\newcommand{\cU}{\mathcal{U}}
\newcommand{\cT}{\mathcal{T}}
\newcommand{\sm}{\mathsf{m}}

\newcommand{\E}{\mathbbm{E}}

\renewcommand{\P}{\mathbbm{P}}

\newcommand{\hH}{\mathbbm{H}}
\newcommand{\CR}{\mathrm{CR}}
\newcommand{\bub}{\mathrm{Bub}}
\DeclareMathOperator{\SLE}{SLE}
\DeclareMathOperator{\CLE}{CLE}
\def\cM{\mathcal{M}}
\newcommand{\dist}{\mathrm{dist}}

\newcommand{\lp}{\mathrm{loop}}

\newcommand{\bx}{\mathbf{x}}
\newcommand{\wt}{\widetilde}

\newtheorem{theorem}{Theorem}[section]
\newtheorem{definition}[theorem]{Definition}
\newtheorem{lemma}[theorem]{Lemma}

\newtheorem{corollary}[theorem]{Corollary}

\newtheorem{proposition}[theorem]{Proposition}
\newtheorem{remark}[theorem]{Remark}

\newcommand\wh[1]{\widehat{#1}}
\newcommand\ol[1]{\overline{#1}}

\def\@rst #1 #2other{#1}
\newcommand\MR[1]{\relax\ifhmode\unskip\spacefactor3000 \space\fi
  \MRhref{\expandafter\@rst #1 other}{#1}}
\newcommand{\MRhref}[2]{\href{http://www.ams.org/mathscinet-getitem?mr=#1}{MR#2}}

\def\MR#1{\href{http://www.ams.org/mathscinet-getitem?mr=#1}{MR#1}}

\begin{document}

\maketitle

\begin{abstract}
We derive the boundary four-point Green's functions for conformal loop ensembles (CLE) with $\kappa\in(4,8)$. Specializing to $\kappa=6$ and $\kappa=16/3$, we establish the exact formulas for the boundary four-point connectivities in critical Bernoulli percolation and the FK-Ising model conjectured by Gori-Viti (2017, 2018). In particular, we identify a logarithmic singularity in the critical FK-Ising model. Our approach also applies to the one-bulk and two-boundary connectivities of CLE, thereby extending the factorization formula of Beliaev-Izyurov (2012) to all $\kappa\in(4,8)$.
\end{abstract}

\setcounter{tocdepth}{1}
\begin{spacing}{0.8}
\small{\tableofcontents}
\end{spacing}
\vspace{10pt}

\section{Introduction}

For over two decades, Schramm–Loewner evolution (SLE)~\cite{schramm0} and its loop counterpart, the conformal loop ensemble (CLE)~\cite{Sheffield2006ExplorationTA,shef-werner-cle}, have played a central role in the study of scaling limits of two-dimensional critical models. The family $\SLE_\kappa$, indexed by $\kappa>0$, describes random fractal curves that arise as scaling limits of interfaces in critical lattice models, while $\CLE_\kappa$ describes collections of loops that encode their full scaling limits, with each loop locally behaving like an $\SLE_\kappa$.

In this paper, we focus on the regime $\kappa \in (4,8)$, which is conjectured to describe the scaling limits of critical FK-$q$ percolation models with $q\in(0,4)$, where $\sqrt{q} = -2\cos(4\pi/\kappa)$. For Bernoulli percolation ($q=1$), Smirnov’s seminal work~\cite{smirnov-cardy} established conformal invariance and convergence to $\SLE_6$~\cite{smirnov-cardy,camia-newman-07}, and to $\CLE_6$ for the full scaling limit~\cite{camia-newman-06}. For the FK-Ising model ($q=2$), the scaling limit has also been proved to be $\CLE_{16/3}$~\cite{CDHKS,Kemppainen-2019-1,kemppainen2019}.

Beyond scaling limits, SLE and CLE provide powerful tools for uncovering the rich integrable structures in critical planar lattice models, and many results that once appeared mysterious at the discrete level can now be established through SLE/CLE. Notable examples include percolation arm exponents~\cite{smirnov-werner-percolation,LSW-one-arm}, left-passage probabilities~\cite{schramm01}, Watts’ crossing formula~\cite{dubedat-watts}, and conformal radii of CLE~\cite{ssw-radii}. Such integrable structures also emerge in the study of multiple SLEs and systems of second-order Belavin--Polyakov--Zamolodchikov (BPZ) equations~\cite{bpz-conformal-symmetry}, which yield connection probabilities for multiple interfaces in various critical models; see e.g.~\cite{PW19,LPW21,FPW22,feng2024multiplesle}. Exact solvabilities for SLE/CLE have also been obtained through their coupling with Liouville quantum gravity (LQG)~\cite{wedges,AHS-SLE-integrability}, including derivations of the backbone exponent~\cite{nolin2024} and three-point connectivity~\cite{acsw24b} in critical percolation.

While two- and three-point observables are now well understood~\cite{acsw24b}, four-point connectivities exhibit nontrivial dependence on the conformal cross-ratio. In this paper, we derive boundary four-point connectivities for $\CLE_\kappa$ with $\kappa \in (4,8)$. These quantities are the (conjectural) scaling limits of the probabilities that four marked boundary points are connected according to one of the three possible link patterns $\{(1234),(12)(34),(14)(23)\}$ in critical FK-$q$ percolation with free boundary conditions; see~\eqref{eq:connectivity-def} and~\eqref{eq:four-type-gf} for precise definitions. For special values of $\kappa$ corresponding to integer $q$, explicit formulas were previously conjectured by Gori and Viti~\cite{GV,Gori:2018gqx} from a conformal field theory (CFT) perspective.

In the four-point setting, conformal automorphisms cannot fix all marked points simultaneously, so the techniques from SLE/LQG coupling used in the three-point case~\cite{acsw24b} no longer apply. Moreover, since our problem concerns boundary connectivities in CLE rather than configurations involving multiple interfaces, the methods of multiple SLE and BPZ equations developed in~\cite{PW19} are not directly applicable. Our approach proceeds in four steps. First, we express the CLE boundary four-point connectivities in terms of boundary Green’s functions of the SLE bubble measure. Second, we prove that the boundary Green’s functions of chordal SLE satisfy a second-order BPZ equation, as expected in~\cite{fakhry2023}. Since the SLE bubble measure arises as a limit of chordal SLE, its boundary Green’s functions are limits of solutions to this equation. Third, we apply Dubédat’s fusion framework~\cite{dubedat2015} to derive a third-order differential equation satisfied by these limiting objects, namely the four-point connectivities; see Theorem~\ref{thm:ode}. This framework, originating in CFT to produce higher-order BPZ equations, was rigorously formulated in~\cite{dubedat2015} in the SLE setting using PDE and representation-theoretic methods. Finally, for each link pattern, we identify the corresponding solution to this third-order equation.

The final step constitutes the main technical novelty of the paper; see Theorem~\ref{thm:identify}. The third-order differential equation admits a three-dimensional space of solutions, corresponding to three distinct Frobenius series. For the link patterns $(12)(34)$ (resp.~$(14)(23)$), the leading-order asymptotics as the cross-ratio tends to $1$ (resp.~$0$) are governed by boundary three-arm exponents, and thus allow for direct identifications with the higher-order Frobenius series. For the remaining link pattern $(1234)$, however, the leading order corresponds to the lowest-order Frobenius series, and it is not a priori clear whether and how the other two Frobenius series may appear as subleading terms.
To resolve this ambiguity, we carry out a refined analysis of the subleading asymptotics, based on the observation that the subleading term of the CLE partition function defined in~\cite{MW18} with two wired boundary arcs decays rapidly. This enables us to determine the higher-order contributions and to uniquely identify each boundary connectivity with a specific solution to the third-order equation.

We first present our results in the setting of critical percolation, confirming the conjectural formula in~\cite{Gori:2018gqx}. We then extend the results to general $\CLE_\kappa$ with $\kappa \in (4,8)$ in Section~\ref{sec:intro-cle}. In particular, for $\kappa = \frac{16}{3}$ (corresponding to the FK-Ising model), we identify a logarithmic singularity in the boundary four-point connectivities, reflecting the underlying logarithmic CFT structure and confirming the conjectures of~\cite{GV}. In Section~\ref{sec:discussion}, we summarize the main ideas of the proof and discuss further applications. In particular, our approach also applies to one-bulk and two-boundary correlation functions; in this setting, we extend the factorization formula of~\cite{KSZ-connectivity,BI12} for $\kappa=6$ to all $\kappa \in (4,8)$ (see Theorem~\ref{thm:bulk}).

\subsection{Boundary four-point connectivities of Bernoulli percolation}\label{sec:intro-perc}
Let $\delta\T\cap\hH$ be the triangular lattice on the upper half-plane with mesh size $\delta$. For $(x_i)_{1\le i\le4}\in\R$, we say they are in counterclockwise order if there exists $k\in\{1,2,3,4\}$ such that $x_{k}<x_{k+1}<x_{k+2}<x_{k+3}$ (we define $x_{i+4}:=x_i$ for $1\le i\le4$).
For $(x_i)_{1\le i\le4}$ in counterclockwise order, let $x_i^\delta$ be an approximation to $x_i$ on $\delta\T\cap\hH$. Consider the critical Bernoulli percolation on $\delta\T\cap\hH$, and denote its law by $\P^\delta$. Define the connectivities of the three distinct link patterns of $(x_i^\delta)_{1\le i\le4}$ by
\begin{equation}\label{eq:connectivity-def}
\begin{aligned}
P^{(1234)}(x_1,x_2,x_3,x_4)=\lim_{\delta\to0}\delta^{-\frac{4}{3}}\P^\delta[x_1^\delta\leftrightarrow x_2^\delta\leftrightarrow x_3^\delta\leftrightarrow x_4^\delta],\\
P^{(12)(34)}(x_1,x_2,x_3,x_4)=\lim_{\delta\to0}\delta^{-\frac{4}{3}}\P^\delta[x_1^\delta\leftrightarrow x_2^\delta\not\leftrightarrow x_3^\delta\leftrightarrow x_4^\delta],\\
P^{(14)(23)}(x_1,x_2,x_3,x_4)=\lim_{\delta\to0}\delta^{-\frac{4}{3}}\P^\delta[x_1^\delta\leftrightarrow x_4^\delta\not\leftrightarrow x_2^\delta\leftrightarrow x_3^\delta].
\end{aligned}
\end{equation}
The existence of these limits is proved in~\cite[Theorem 1.9]{cf2025}, and the normalization factor $\delta^{-\frac{4}{3}}$ is from the percolation boundary one-arm exponent $\frac{1}{3}$. Note that when $(x_i)_{1\le i\le4}$ is in counterclockwise order, one cannot have $\{x_1^\delta\leftrightarrow x_3^\delta\not\leftrightarrow x_2^\delta\leftrightarrow x_4^\delta\}$.
We also let
\[
P^{\rm total}(x_1,x_2,x_3,x_4):=P^{(1234)}(x_1,x_2,x_3,x_4)+P^{(12)(34)}(x_1,x_2,x_3,x_4)+P^{(14)(23)}(x_1,x_2,x_3,x_4).
\]

Our first main result, Theorem~\ref{thm:percolation}, gives exact expressions for these limits, thereby proving the formulas conjectured in~\cite{Gori:2018gqx}. Let $(x_i)_{1\le i\le4}$ be as above, and
\[
\lambda:=\frac{(x_2-x_1)(x_4-x_3)}{(x_4-x_2)(x_3-x_1)}\in(0,1)\]
be the cross-ratio of $(x_i)_{1\le i\le4}$. Define two functions
\begin{align*}
F_L(\lambda) &= (\lambda (1-\lambda))^{\frac{4}{9}} \, {}_3F_2\!\left(-\frac{2}{9},-\frac{1}{18},\frac{7}{9};\frac{1}{3},\frac{2}{3};\frac{4}{27}\frac{(\lambda^2-\lambda+1)^3}{(1-\lambda)^2 \lambda^2}\right), \\
F_S(\lambda) &= (1-\lambda)^2 \lambda^2 \, {}_3F_2\!\left(\frac{4}{3},\frac{3}{2},\frac{7}{3};\frac{8}{3},3;4 \lambda (1-\lambda)\right).
\end{align*}
Here $_3F_2$ is the generalized hypergeometric function. For $0<\lambda\le\frac{1}{2}$, let $V_2(\lambda):=F_S(\lambda)$, and there exist two real-valued functions $V_0(\lambda), V_{1/3}(\lambda)$ such that
\[
F_L(\lambda)=- e^{\frac{2\pi i}{9}}\frac{8\pi^{\frac{3}{2}} 2^{\frac{5}{9}} \sin\!\left(\frac{4\pi}{9}\right) 3^{\frac{5}{6}}}{\Gamma\!\left(\frac{5}{6}\right) \Gamma\!\left(-\frac{1}{18}\right) ^2\Gamma\!\left(\frac{7}{9}\right)} V_0(\lambda)+ e^{\frac{\pi i}{18}}\frac{8 \,\Gamma\!\left(\frac{5}{6}\right)^2 2^{\frac{1}{9}} 3^{\frac{1}{3}} \pi \, }{9 \,\Gamma\!\left(\frac{7}{18}\right) \Gamma\!\left(\frac{13}{18}\right) \Gamma\!\left(\frac{7}{9}\right)^2}V_{1/3}(\lambda).
\]
Then we have $V_0(\lambda)=1-\frac{2}{3}\lambda+\frac{8}{45}\lambda^2|\log\lambda|+O(\lambda^2)$, $V_{1/3}(\lambda)=\lambda^{\frac{1}{3}}(1-\frac{1}{2}\lambda+O(\lambda^2))$, and $V_2(\lambda)=\lambda^2(1+\frac{1}{3}\lambda+O(\lambda^2))$ as $\lambda\downarrow0$. For $\lambda\in(\frac{1}{2},1)$, we define
\[
V_0(\lambda)=V_0(1-\lambda),\quad V_2(\lambda)=\alpha{\rm Im} F_L(1-\lambda)+F_S(1-\lambda)
\]
with $\alpha:=\frac{405}{64}  \,\pi^{-\frac{7}{2}}\,\Gamma\!\left(\frac{7}{9}\right)^2 \Gamma\!\left(-\frac{1}{18}\right) \Gamma\!\left(\frac{2}{3}\right)^3 2^{\frac{8}{9}}$. Then $V_0,V_2$ are smooth real-valued functions on $(0,1)$. We refer readers to~\cite[Appendix C]{Gori:2018gqx} for further details of these functions.

\begin{theorem}\label{thm:percolation}
Let $V_0,V_2$ be defined as above. There exists a constant $C\in(0,\infty)$ such that
\begin{align*}
P^{\rm total}(x_1,x_2,x_3,x_4)&= C\left(\frac{(x_4-x_2)(x_3-x_1)}{(x_2-x_1)(x_4-x_3)(x_3-x_2)(x_4-x_1)}\right)^{\frac{2}{3}}V_0(\lambda),\\
P^{(14)(23)}(x_1,x_2,x_3,x_4)&= AC\left(\frac{(x_4-x_2)(x_3-x_1)}{(x_2-x_1)(x_4-x_3)(x_3-x_2)(x_4-x_1)}\right)^{\frac{2}{3}}V_2(\lambda).
\end{align*}
Here, the constant $A:=\frac{8\sqrt{3}\,\pi \sin\!\left(\frac{2\pi}{9}\right)}{135\cos\!\left(\frac{5\pi}{18}\right)}\in(0,\infty)$.
\end{theorem}
By symmetry, we have $P^{(12)(34)}(x_1,x_2,x_3,x_4)=P^{(14)(23)}(x_4,x_1,x_2,x_3)$. Thus, Theorem~\ref{thm:percolation} also implies $P^{(12)(34)}(x_1,x_2,x_3,x_4)= AC\left(\frac{(x_4-x_2)(x_3-x_1)}{(x_2-x_1)(x_4-x_3)(x_3-x_2)(x_4-x_1)}\right)^{\frac{2}{3}}V_2(1-\lambda)$.
Consequently, Theorem~\ref{thm:percolation} gives the exact forms of the three limits in~\eqref{eq:connectivity-def}.
Furthermore, when expanding at $x_2\to x_1$ (with $x_1,x_3,x_4$ fixed; hence $\lambda\to0$), we have
\begin{align*}
P^{\rm total}(x_1,x_2,x_3,x_4)
&=C(x_2-x_1)^{-\frac{2}{3}}(x_4-x_3)^{-\frac{2}{3}}\left(1+\lambda^2\left(\frac{8}{45}|\log\lambda|+\frac{16}{25}\right)+O(\lambda^3|\log\lambda|)\right).
\end{align*}
Therefore, the universal constant $C_2^\hH$ in~\cite[Theorem 1.9]{cf2025} equals $\frac{8}{45}$.

As a corollary, we have the following exact form of the \emph{universal ratio} $\frac{P^{(14)(23)}(x_1,x_2,x_3,x_4)}{P^{\rm total}(x_1,x_2,x_3,x_4)}$ introduced in~\cite{Gori:2018gqx}. The term ``universal" refers to the expectation that this ratio is independent of the lattice.
\begin{corollary}
Let $(x_i)_{1\le i\le4}$ be in counterclockwise order, and $\lambda=\frac{(x_2-x_1)(x_4-x_3)}{(x_4-x_2)(x_3-x_1)}$ be its cross-ratio. Then
\begin{align*}
\frac{P^{(14)(23)}(x_1,x_2,x_3,x_4)}{P^{\rm total}(x_1,x_2,x_3,x_4)}=A\frac{V_2(\lambda)}{V_0(\lambda)}=A\lambda^2\left(1+\lambda+\lambda^2\left(-\frac{8}{45}|\log\lambda|+\frac{1607}{4950}\right)+O(\lambda^3|\log\lambda|)\right)
\end{align*}
as $\lambda\to0$. Here $A$ is the same constant as in Theorem~\ref{thm:percolation}.
\end{corollary}

\subsection{Boundary four-point connectivities of CLE}\label{sec:intro-cle}

We now state our result for general $\CLE_\kappa$ with $\kappa\in(4,8)$, from which Theorem~\ref{thm:percolation} follows by taking $\kappa=6$. $\CLE_\kappa$ is a random collection of loops satisfying the domain Markov property and conformal invariance, and each loop is an $\SLE_\kappa$ curve~\cite{Sheffield2006ExplorationTA,shef-werner-cle}. When $\kappa\in(4,8)$, $\CLE_\kappa$ loops are self-touching and can touch each other. In the following, we fix
\[
\kappa\in(4,8),\quad h:=\frac{8}{\kappa}-1\in(0,1).
\]

Let $\Gamma$ be a (non-nested) $\CLE_\kappa$ on the upper half-plane $\hH$, and let $\cT(\Gamma)$ be the collection of loops in $\Gamma$ that touch the boundary $\partial\hH=\R$. For each loop $\ell\in\cT(\Gamma)$, denote $\nu_{\ell\cap\R}$ to be the $(1-h)$-dimensional Minkowski content measure of $\ell\cap\R$, such that for any open interval $J\subset\R$,
\begin{equation}\label{eq:def-mink}
\nu_{\ell\cap\R}(J):=\lim_{\varepsilon\to0}\varepsilon^{-h}{\rm Leb}_\R(\{x\in J: {\rm \dist}(x,\ell\cap \R)<\varepsilon\}).
\end{equation}
The existence of $\nu_{\eta\cap\R}$ is proved by~\cite{lawler-mink-R} (see also~\cite{zhan-boundary-gf}) and the local absolute continuity between $\CLE_\kappa$ and $\SLE_\kappa$. 

For $(x_i)_{1\le i\le 4}\in\R$ in counterclockwise order, similar to~\eqref{eq:connectivity-def}, we define three types of boundary four-point Green's functions of $\CLE_\kappa$ as
\begin{equation}\label{eq:four-type-gf}
\begin{aligned}
&G^{(1234)}(x_1,x_2,x_3,x_4)dx_1dx_2dx_3dx_4=\E\left[\sum_{\ell\in\cT(\Gamma)}\prod_{i=1}^4\nu_{\ell\cap\R}(dx_i)\right],\\
&G^{(12)(34)}(x_1,x_2,x_3,x_4)dx_1dx_2dx_3dx_4=\E\left[\sum_{\ell,\ell'\in\cT(\Gamma)}{\bf 1}_{\ell\neq\ell'}\prod_{i=1}^2\nu_{\ell\cap\R}(dx_i)\prod_{j=3}^4\nu_{\ell'\cap\R}(dx_j)\right],\\
&G^{(14)(23)}(x_1,x_2,x_3,x_4)=G^{(12)(34)}(x_4,x_1,x_2,x_3).
\end{aligned}
\end{equation}
Here the expectation $\E$ is with respect to the $\CLE_\kappa$ configuration $\Gamma$ on $\hH$.
We also let
\[
G^{\rm total}(x_1,x_2,x_3,x_4)=G^{(1234)}(x_1,x_2,x_3,x_4)+G^{(12)(34)}(x_1,x_2,x_3,x_4)+G^{(14)(23)}(x_1,x_2,x_3,x_4).
\]
We also refer to~\eqref{eq:four-type-gf} as boundary four-point connectivities of $\CLE_\kappa$ since they are the counterparts of~\eqref{eq:connectivity-def} in the critical FK-$q$ percolation, whose scaling limit is conjectured to be $\CLE_\kappa$ with $\sqrt{q} = -2\cos(4\pi/\kappa)$.

Let $\sP$ be the collection of link patterns, i.e.
\[
\sP:=\{(1234),(12)(34),(14)(23),{\rm total}\}.
\]
Note that these boundary Green's functions satisfy the conformal covariance
\begin{equation*}
G^p(x_1,x_2,x_3,x_4)=\left(\prod_{i=1}^4|\phi'(x_i)|^{h}\right)\cdot G^p(\phi(x_1),\phi(x_2),\phi(x_3),\phi(x_4)),\quad p\in\sP
\end{equation*}
for any M\"obius transformation $\phi$ on $\hH$. Therefore, for each $p\in\sP$, there exists a function $U^p:(0,1)\to\R_+$ of the cross-ratio $\lambda=\frac{(x_2-x_1)(x_4-x_3)}{(x_4-x_2)(x_3-x_1)}\in(0,1)$ such that
\begin{equation}\label{eq:u-lambda}
G^p(x_1,x_2,x_3,x_4)=\left(\frac{(x_4-x_2)(x_3-x_1)}{(x_2-x_1)(x_4-x_3)(x_3-x_2)(x_4-x_1)}\right)^{2h}U^p (\lambda).
\end{equation}
Note that by symmetry, we have $U^{(14)(23)} (\lambda)=U^{(12)(34)}(1-\lambda)$.

The following result gives the differential equation satisfied by $U^p$.
\begin{theorem}\label{thm:ode}
For $\kappa\in(4,8)$ and $p\in\sP$, $U^p(\lambda)$ is smooth and solves the following third-order ODE
\begin{equation}\label{eq:ode}
\begin{aligned}
\frac12 \kappa^3\lambda^2 (1 - \lambda)^2U''' &+ \kappa^2\lambda(3\kappa - 16)(1 - \lambda)(1 - 2\lambda)U'' \\
&+ \kappa\left[ 3(\kappa - 4)(\kappa - 8) + \lambda(\lambda - 1)(18\kappa^2 - 212\kappa + 608) \right]U' \\&+ 6(2\lambda - 1)(\kappa - 4)(\kappa - 8)^2 U = 0.
\end{aligned}
\end{equation}
\end{theorem}

Note that the third-order ODE~\eqref{eq:ode} also appeared in~\cite[Eq.(3.8)]{Gori:2018gqx}, where the $\CLE_\kappa$ boundary Green's functions~\eqref{eq:four-type-gf} are interpreted from the CFT perspective as correlation functions of the primary field $\phi_{1,3}$ with a level-three null vector.

For $\kappa\in(4,8)$, no closed-form solutions of~\eqref{eq:ode} are known in general. Note that $0$ is a regular singular point of~\eqref{eq:ode}, with indicial roots $0,h,3h+1$.
Then, in a neighborhood of $0$, we can use the Frobenius method to find three linearly independent solutions $V_0,V_h,V_{3h+1}$ of~\eqref{eq:ode} such that
\begin{equation}\label{def:v}
V_0(\lambda)=1+O(\lambda),\quad V_h(\lambda)=\lambda^h(1+O(\lambda)),\quad V_{3h+1}(\lambda)=\lambda^{3h+1}(1+O(\lambda)).
\end{equation}
Since~\eqref{eq:ode} only has two singular points $0$ and $1$, we know that the three Frobenius series $V_0,V_h,V_{3h+1}$ converge for all $\lambda\in(0,1)$.
When $h$ and $3h+1$ are both non-integers, the expressions $1+O(\lambda)$ above are indeed power series of $\lambda$. When $h$ or $3h+1$ is integer (for $\kappa\in(4,8)$, this only occurs when $\kappa=6,\frac{16}{3},\frac{24}{5}$), there might be logarithmic terms in $O(\lambda)$, and we find explicit solutions in these cases separately.
We refer readers to~\cite[Chapter XVI]{ODE} for more background on the Frobenius method.

The next theorem shows that the three types of boundary Green's functions~\eqref{eq:four-type-gf} can be uniquely expressed in terms of the three linearly independent solutions $V_0,V_h,V_{3h+1}$~\eqref{def:v} of~\eqref{eq:ode}.
\begin{theorem}\label{thm:identify}
There exist constants $C_1,C_2\in(0,\infty)$ such that for $\lambda\in(0,1)$, $U^{(14)(23)}(\lambda)=C_1V_{3h+1}(\lambda)$, $U^{(12)(34)}(\lambda)=U^{(14)(23)}(1-\lambda)=C_1V_{3h+1}(1-\lambda)$, and $U^{\rm total}(\lambda)=C_2(V_0(\lambda)+\beta V_{3h+1}(\lambda))$. Here $\beta\in\R$ is the unique value with $V_0(1-\lambda)+\beta V_{3h+1}(1-\lambda)-1=o(\lambda^h)$, and the ratio $\frac{C_1}{C_2}$ is determined by $\frac{C_1}{C_2}V_{3h+1}(1-\lambda)\to1$ as $\lambda\to0$.
\end{theorem}

As we have emphasized before, establishing $U^{\rm total}(\lambda)=1+o(\lambda^h)$ as $\lambda\to0$ requires a delicate analysis of boundary Green's functions~\eqref{eq:four-type-gf}; see also the discussions in Section~\ref{sec:discussion}.

For $\kappa=6$ and $h=\frac{1}{3}$, the functions $V_0,V_{1/3},V_2$ introduced in Section~\ref{sec:intro-perc} exactly correspond to the three solutions~\eqref{def:v}. Using a basic one-arm coupling in~\cite{conijn15}, one can show that the connectivities defined in~\eqref{eq:connectivity-def} agree with the $\CLE_6$ boundary four-point Green's functions in~\eqref{eq:four-type-gf}; see Appendix~\ref{appendix:discrete} for details.
Theorem~\ref{thm:percolation} is then the $\kappa=6$ case of Theorem~\ref{thm:ode} and~\ref{thm:identify}.

As another special case, taking $\kappa=\frac{16}{3}$ gives the boundary four-point connectivities of the critical FK-Ising model. Consider the critical FK-Ising model on $\delta\Z^2\cap\hH$ with free boundary condition. In analogy with the Bernoulli percolation, the limit $P^{(1234)}_{\rm{FK}}(x_1,x_2,x_3,x_4)=\lim_{\delta\to0}\delta^{-2}\P^\delta_{\rm{FK}}[x_1^\delta\leftrightarrow x_2^\delta\leftrightarrow x_3^\delta\leftrightarrow x_4^\delta]$ exists~\cite[Theorem 1.4]{CF-FK}, and agrees with the $\CLE_{16/3}$ boundary four-point Green's function (see also Appendix~\ref{appendix:discrete}). For other $p\in\sP$, $P^p_{\rm{FK}}(x_1,x_2,x_3,x_4)$ is defined similarly. By solving~\eqref{eq:ode} with $\kappa=\frac{16}{3}$, these connectivities are explicit; see Section~\ref{sec:kappa=16/3}. In particular, we have
\begin{theorem}\label{thm:fk}
Let $R_{\rm{FK}}(\lambda)=\frac{P^{(12)(34)}_{\rm{FK}}(x_1,x_2,x_3,x_4)}{P^{\rm total}_{\rm{FK}}(x_1,x_2,x_3,x_4)}$, and for $x\in(0,1)$, define
\begin{equation}\label{eq:fk-g}
g(x) = -\frac{(1 - x)^{3/2} \left( (2 - 4x) \, {}_2F_1\!\left(\frac{3}{2}, \frac{7}{2}; 3; 1 - x\right) - 3x(1 - x) \, {}_2F_1\!\left(\frac{5}{2}, \frac{9}{2}; 4; 1 - x\right) \right) x^{3/2}}{2\left(x + (1 - x)^2\right)^2}.
\end{equation}
Then we have $R_{\rm{FK}}(\lambda)=A_{\rm{FK}}\int_0^\lambda g(x)dx$, with $A_{\rm{FK}}=(\int_0^1 g(x)dx)^{-1}\approx1.19948$.
\end{theorem}
Theorem~\ref{thm:fk} proves the conjectural formula in~\cite[Eq.(6)]{GV}. Remarkably, taking $\lambda=1-\varepsilon$ with $\varepsilon\to0$ (i.e.~$x_2\to x_3$ with $x_1,x_3,x_4$ fixed), we have
\begin{equation*}
R_{\rm{FK}}(1-\varepsilon)=1-A_{\rm{FK}}\varepsilon^{\frac{1}{2}}\left( 
\frac{64}{21\pi}
+\frac{16}{21\pi}\varepsilon
-\frac{941-840\log2-210|\log\varepsilon|}{525\pi}\varepsilon^2
+ O\bigl(\varepsilon^3|\log \varepsilon|\bigr)\right).
\end{equation*}
See also~\cite[Eq.(S12)]{GV}. To our knowledge, this provides the first rigorous evidence that a logarithmic singularity appears in the correlation functions of the critical FK-Ising model.

\begin{remark}
In~\cite[Conjecture 1.2]{feng2024multiplesle}, the authors formulate a precise conjecture for the limiting connection probabilities of general critical loop $O(n)$ models, which also give the conjectural connection probabilities for multichordal $\CLE_\kappa$ with $2N$ boundary arcs (see~\cite{ambrosio2025}). In the regime $\kappa\in(4,8)$, this conjecture was proved for $n=\sqrt{2}$ (FK-Ising)~\cite{FPW22} and for $2N=4$~\cite{MW18}. In principle, by taking $2N=8$ and shrinking four of the eight boundary arcs, one could obtain the CLE boundary Green's functions~\eqref{eq:four-type-gf} via renormalized limits of the corresponding conjectural connection probabilities. However, this does not seem practically feasible, since the explicit formulas for these conjectural connection probabilities involve the pure partition functions, whose constructions are highly nontrivial~\cite{FK15,KP20,kp16,wu-hsle,Zhan-multiple,feng2024multiplesle}. It is therefore unclear whether our results can be recovered in this way from these conjectural formulas.
\end{remark}

\subsection{Discussions}\label{sec:discussion}
We now briefly discuss our proof strategy. The first step is to identify a precise correspondence between $\CLE_\kappa$ and the $\SLE_\kappa$ bubble measure for $\kappa\in(4,8)$. To achieve this, we show that the loop sampled from the counting measure on boundary-touching loops in a $\CLE_\kappa$ configuration is equal in law to the unrooted $\SLE_\kappa$ bubble measure, see Proposition~\ref{prop:touching-equals-bubble}. This identification allows us to express the boundary Green's functions of $\CLE_\kappa$~\eqref{eq:four-type-gf} in terms of those of the $\SLE_\kappa$ bubble measure. Furthermore, $\SLE_\kappa$ bubble measure is also the weak limit of the chordal $\SLE_\kappa$ when its starting point approaches its end point. According to~\cite{fakhry2023}, the boundary Green's functions of chordal $\SLE_\kappa$ are finite and satisfy a martingale property as the chordal $\SLE_\kappa$ grows. Combined with their smoothness, the martingale property then yields a pair of second-order PDEs, known as BPZ equations. Establishing such smoothness is nontrivial and relies on H\"ormander's hypoellipticity; see Lemma~\ref{lem:smooth}, as inspired by~\cite{dub15localization,fakhry2023}.
Then, we apply a fusion procedure to these PDEs to derive a third-order differential equation satisfied by boundary Green's functions of the $\SLE_\kappa$ bubble, which corresponds to collapsing two marked boundary points. As mentioned before, this step follows from the framework in~\cite{dubedat2015}, which also appeared in recent works~\cite{peltola-c--2,peltola-c-1} for specific values of $\kappa$. This proves Theorem~\ref{thm:ode}.

Theorem~\ref{thm:identify} requires a delicate asymptotic analysis for the boundary Green's functions defined in~\eqref{eq:four-type-gf}. Its crucial ingredient is Proposition~\ref{prop:identify-v0}, which establishes the $o(\lambda^h)$ subleading order in the sum of the Green's functions $G^{(1234)}(x_1,x_2,x_3,x_4)+G^{(12)(34)}(x_1,x_2,x_3,x_4)$ as the cross-ratio $\lambda\to0$. The proof of Proposition~\ref{prop:identify-v0} is arguably the most subtle part of this paper. In particular, it involves re-expressing the Green's functions via a careful decomposition of boundary-touching $\CLE_\kappa$ loops (see Corollary~\ref{cor:12-34}), as well as noting the rapid subleading decay in the partition function of $\CLE_\kappa$ with two wired boundary arcs, based on its connection probability~\cite{MW18} (see~\eqref{eq:z-tau}). According to~\cite[Lemma 6.1]{feng2024multiplesle}, such rapid decay of subleading terms also holds for Coulomb gas integrals, which are conjectured to be the partition functions of general multichordal $\CLE_\kappa$. Theorems~\ref{thm:percolation} and~\ref{thm:fk} then follow by taking $\kappa=6$ and $\kappa=\frac{16}{3}$.

Our approach to the boundary four-point Green's functions can be extended to the one-bulk and two-boundary connectivities of $\CLE_\kappa$.
Let $x_1,x_2\in\R$, $z\in\hH$, and consider the Green's function
\begin{equation}\label{eq:def-gf-bulk}
G(x_1,x_2,z)dx_1dx_2dz=\E\left[\sum_{\ell\in\cT(\Gamma)}\Upsilon_\ell(dz)\prod_{i=1}^2\nu_{\ell\cap\R}(dx_i)\right].
\end{equation}
Here, $\Upsilon_\ell$ denotes the \emph{Miller-Schoug measure}~\cite{miller2023existence} of the $\CLE_\kappa$ gasket surrounded by $\ell$. The Miller-Schoug measure is a canonical measure supported on the $\mathrm{CLE}_\kappa$ gasket, characterized by conformal covariance and the domain Markov property. It is conjectured to coincide with the $(2-\alpha)$-dimensional Minkowski content of the $\mathrm{CLE}_\kappa$ gasket, where $\alpha:=\frac{(3\kappa-8)(8-\kappa)}{32\kappa}$. For $\CLE_6$ gasket, its Miller-Schoug measure is shown to be the scaling limit of the normalized counting measure on the critical percolation cluster~\cite{gps-pivotal,CL24}. Consequently, $G(x_1,x_2,z)$ can be interpreted as the normalized probability that $x_1$, $x_2$, and $z$ belong to the same cluster.
\begin{theorem}\label{thm:bulk}
For $\kappa\in(4,8)$, let $h=\frac{8}{\kappa}-1$ and $\alpha$ be as above. Then
\begin{equation}\label{eq:bulk-final}
G(x_1,x_2,z)=C|x_2-x_1|^{-h}{\rm Im}(z)^{h-\alpha}|z-x_1|^{-h}|z-x_2|^{-h}
\end{equation}
for some constant $C\in(0,\infty)$.
\end{theorem}
For the case of critical Bernoulli percolation (i.e.~$\kappa=6$), Theorem~\ref{thm:bulk} recovers the factorization formula conjectured by~\cite{KSZ-connectivity} and later proved by~\cite{BI12}: the square of the one-bulk and two-boundary connectivity at $z\in \hH, x_1, x_2 \in \R$ factorizes as the product of two bulk-boundary connectivities at $(z, x_1)$ and $(z, x_2)$ and the boundary two-point connectivity at $(x_1, x_2)$. Theorem~\ref{thm:bulk} extends this remarkably simple structure to all $\kappa \in (4,8)$, showing that the factorization is in fact a universal feature of $\CLE_\kappa$ and hence of critical FK-$q$ percolation.

However, our approach does not apply to the bulk two-point correlation functions, which were recently conjectured in physics~\cite{downing2026}, as the tools of the Loewner equation are no longer available. The same limitation applies to the four-point functions on the sphere. For the latter, differential equations can be derived from the CFT perspective in some special cases~\cite{Gamsa_2006}, which admit exact closed-form solutions. It would be interesting to explore whether such equations admit a probabilistic interpretation in the SLE context.

\medskip
\noindent\textbf{Organization of the paper.} Section~\ref{sec:bubble&gf} is devoted to expressing the $\CLE_\kappa$ boundary Green's functions in terms of the $\SLE_\kappa$ bubble measure, and showing that the latter are the corresponding limits of chordal $\SLE_\kappa$. Then in Section~\ref{sec:fusion}, we prove Theorem~\ref{thm:ode}. In Section~\ref{sec:solution}, we prove Theorem~\ref{thm:identify} (and hence Theorems~\ref{thm:percolation} and~\ref{thm:fk}). We also include solutions to~\eqref{eq:ode} for other special $\kappa$'s in Section~\ref{sec:kappa-le-4}. In Section~\ref{sec:bulk-boundary}, we prove Theorem~\ref{thm:bulk}. In Appendix~\ref{appendix:gf}, we provide a detailed background on the definitions of SLE boundary Green's functions. Appendix~\ref{appendix:discrete} gives relevant discrete details in the proof of Theorems~\ref{thm:percolation} and~\ref{thm:fk}. We also provide MATLAB code in Appendix~\ref{appendix:matlab} that verifies the calculation of deriving~\eqref{eq:ode} from Section~\ref{sec:fusion}.

\medskip
\noindent\textbf{Acknowledgment.}
The author thanks Xin Sun for his suggestions on the draft of this paper.
This work was supported by the National Natural Science Foundation of China (Grant No.~12526204) and National Key R\&D Program of China (No.~2021YFA1002700).

\section{SLE bubble measures and CLE boundary Green's functions}\label{sec:bubble&gf}

In the following, we fix $\kappa\in(4,8)$ and $h:=\frac{8}{\kappa}-1\in(0,1)$.
For a finite measure $M$, we denote its total mass by $|M|$, and denote $M^\#:=\frac{1}{|M|}M$.
We usually use the superscript \# to indicate that some measure is a probability measure. For a compact set $A\subset\ol\hH$, we denote $\cU(A)$ to be the unbounded connected component of $\hH\setminus A$. For $z\in\C$ and $\varepsilon>0$, we let $B(z,\varepsilon):=\{w\in\C:|w-z|<\varepsilon\}$. For $x\in\R$, we also let $I(x,\varepsilon)$ be the open interval $(x-\varepsilon,x+\varepsilon)$.

We first review the background on the SLE bubble measure in Section~\ref{sec:sle-bubble}, and then show the relation between boundary-touching CLE loops and the (unrooted) SLE bubble measure in Section~\ref{sec:unrooted}. In Section~\ref{sec:boundary-gf}, we relate the CLE boundary four-point connectivities~\eqref{eq:four-type-gf} to the boundary Green's functions of SLE bubble and of chordal SLE. In Section~\ref{sec:smoothness}, we show the smoothness of these boundary Green's functions, thus obtaining a pair of second-order PDEs satisfied by the boundary Green's functions of chordal SLE, which will be the basis of the fusion procedure in Section~\ref{sec:fusion}.

\subsection{SLE bubble measure}\label{sec:sle-bubble}

The $\SLE_\kappa$ bubble measure was first introduced in~\cite{shef-werner-cle}, and systematically studied in~\cite{zhan-bubble}. For $x,y\in\R$, let $\mu_{\hH,x,y}^\#$ be the law of the chordal $\SLE_\kappa$ from $x$ to $y$. Then the $\SLE_\kappa$ bubble measure $\mu_x^\bub$ rooted at $x$ is defined by the weak limit
\begin{equation}\label{eq:bubble-def}
\mu_x^\bub=\lim_{\varepsilon\to0}\varepsilon^{-h}\mu_{\hH,x,x+\varepsilon}^\#.
\end{equation}
For $\gamma$ sampled from $\mu_x^\bub$, the $(1-h)$-dimensional Minkowski content measure $\nu_{\gamma\cap\R}$ of $\gamma\cap\R$ exists~\cite[Theorem 6.17]{zhan-boundary-gf}.
According to~\cite[Theorem 4.8]{zhan-bubble}, we can define the $\SLE_\kappa$ bubble measure $\mu_{x,y}^\bub$ with two marked points $x,y$ such that
\begin{equation}\label{eq:bubble-two-pinned}
\mu_{x,y}^\bub(d\gamma)dy=\nu_{\gamma\cap\R}(dy)\mu_x^\bub(d\gamma).
\end{equation}
Here the total mass of $\mu_{x,y}^\bub(d\gamma)$ is $C|y-x|^{-2h}$, where $C\in(0,\infty)$ is a fixed constant introduced in~\cite[Eq.(58)]{zhan-bubble}.
Such $\mu_x^\bub$ and $\mu_{x,y}^\bub$ corresponds to the one-point and two-point pinned loop measures in~\cite{shef-werner-cle}, respectively. According to~\cite[Section 4.2]{zhan-bubble}, for $x<y$, the normalized probability measure $(\mu_{x,y}^\bub)^\#$ can be decomposed as follows: first sample the chordal $\SLE_\kappa(2)$ curve $\eta$ on $\hH$ from $x$ to $y$ with force point at $x-$, and then sample the chordal $\SLE_\kappa$ curve $\eta'$ on the remaining unbounded domain $\cU(\eta)$ from $y$ to $x$. Then the concatenation of $\eta$ and $\eta'$ has the same law as $(\mu_{x,y}^\bub)^\#$. The marginal law of $\eta'$ is the chordal $\SLE_\kappa(2)$ from $y$ to $x$, with force point at $y-$.

From~\eqref{eq:bubble-def}, we also introduce the \emph{unrooted} $\SLE_\kappa$ bubble measure $\mu^\bub$ via
\begin{equation}\label{eq:unrooted-def}
\mu^\bub(d\gamma)=\frac{1}{|\nu_{\gamma\cap\R}|^2}\int_{\R\times\R}\mu_{x,y}^\bub(d\gamma)dxdy.
\end{equation}
By~\cite[Theorem 4.13]{zhan-bubble}, $\mu^\bub$ then is invariant under M\"obius transforms on $\hH$, with
\begin{align}
\nu_{\gamma\cap\R}(dx)\mu^\bub(d\gamma)&=\mu_x^\bub(d\gamma)dx,\label{eq:rooted-unrooted-1}\\
\nu_{\gamma\cap\R}(dx)\nu_{\gamma\cap\R}(dy)\mu^\bub(d\gamma)&=\mu_{x,y}^\bub(d\gamma)dxdy.\label{eq:rooted-unrooted-2}
\end{align}
We refer readers to~\cite{zhan-bubble} for further background on $\SLE_\kappa$ bubble measures.

The following result from~\cite{shef-werner-cle} links $\CLE_\kappa$ to the $\SLE_\kappa$ bubble measure. The original statement in~\cite{shef-werner-cle} is for simple $\CLE_\kappa$, i.e.~$\kappa\in(\frac{8}{3},4]$. However, since the proof there only uses the conformal Markovian exploration process of $\CLE_\kappa$, which is based on the domain Markov property of $\CLE_\kappa$ and conformal invariance, the result is still valid for $\kappa\in(4,8)$.
\begin{lemma}[{\cite{shef-werner-cle}}]\label{lem:bubble-cle}
Suppose $z\in\hH$, and $\Gamma$ is a $\CLE_\kappa$ on $\hH$. Let $T_\varepsilon(z)$ (resp.~$\wh T_\varepsilon(z)$) be the event that the $\CLE_\kappa$ loop surrounding $z$ intersects $B(0,\varepsilon)$ (resp.~$I(0,\varepsilon)$). Then as $\varepsilon\to0$, both $\P[T_\varepsilon(z)]$ and $\P[\wh T_\varepsilon(z)]$ are $\varepsilon^{h+o(1)}$; conditioned on $T_\varepsilon(z)$ (resp.~$\wh T_\varepsilon(z)$), the conditional law of $\ell$ converges to $\mu_0^\bub$ conditioned to surround $z$.
\end{lemma}

\subsection{Boundary-touching CLE loops and unrooted SLE bubble measure}\label{sec:unrooted}

We now relate the boundary-touching $\CLE_\kappa$ loops to the unrooted $\SLE_\kappa$ bubble measure. 
\begin{proposition}\label{prop:touching-equals-bubble}
Let $\Gamma$ be a $\CLE_\kappa$ configuration on $\hH$, and let $\cT(\Gamma)$ be the collection of loops in $\Gamma$ that touch $\R$. Denote $\rho$ to be the law of the outer boundary of the loop chosen from the counting measure, i.e.~$\rho(\mathcal{A})=\E\left[\sum_{\ell\in\cT(\Gamma)}{\bf 1}_{\ell\in\mathcal{A}}\right]$ for any measurable $\mathcal{A}$. Then there exists ${\frak C}\in(0,\infty)$ such that $\rho$ is equal to the unrooted $\SLE_\kappa$ bubble measure ${\frak C}\mu^\bub$ defined in~\eqref{eq:unrooted-def}.
\end{proposition}
\begin{proof}
Let $\ell_z\in\Gamma$ be the outermost loop with its outer boundary surrounding $z\in\hH$. We claim that $\ell_z$, restricted on that $\ell_z\cap\R\neq\emptyset$, has the same law as the unrooted  $\SLE_\kappa$ bubble measure $\mu^\bub$ restricted on surrounding $z$. The result then follows by varying $z$.

Denote the law of $\ell_z$ by $\Theta_z$. Let $u(\varepsilon)=\P[ \wh T_\varepsilon(i)]=\varepsilon^{h+o(1)}$ where $\wh T_\varepsilon(i)$ is defined in Lemma~\ref{lem:bubble-cle}. For $x\in\R$ and $\varepsilon>0$, let $E_{z,x,\varepsilon}$ be the event that $\ell_z$ intersects $I(x,\varepsilon)$. Now, consider the measure $\mathbbm{M}_\varepsilon:=u(\varepsilon)^{-1}{\bf 1}_{E_{z,x,\varepsilon}}\Theta_z(d\ell)dx$. Note that for fixed $x\in\R$, Lemma~\ref{lem:bubble-cle} implies that the measure $u(\varepsilon)^{-1}{\bf 1}_{E_{z,x,\varepsilon}}\Theta_z$ converges weakly to $C{\bf 1}_{F_{z,x}}\mu_x^\bub$ as $\varepsilon\to0$, where $F_{z,x}$ is the event that the bubble rooted at $x$ surrounds $z$ and $C\in(0,\infty)$ is some constant (throughout this proof, $C$ stands for some constant whose value can vary from line to line). Furthermore, by conformal covariance, the convergence is uniform on any bounded interval $U\subset\R$. Thus, we have the vague convergence
\begin{equation}\label{eq:m-eps-1}
    \mathbbm{M}_\varepsilon\to\mathbbm{M}:=C{\bf 1}_{F_{z,x}}\mu_x^\bub(d\gamma)dx=C{\bf 1}_{\gamma\text{ surrounds } z}\nu_{\gamma\cap\R}(dx)\mu^\bub(d\gamma)
\end{equation}
as $\varepsilon\to0$. Here the last equality is due to~\eqref{eq:rooted-unrooted-1}. On the other hand, note that for a.s.~$\ell$ sampled from $\Theta_z(d\ell)$ such that $\ell\cap\R\neq\emptyset$, by~\cite[Theorem 6.17]{zhan-boundary-gf}, $\varepsilon^{-h}{\bf 1}_{E_{z,x,\varepsilon}}dx$ weakly converges to the $(1-h)$-dimensional Minkowski content $\nu_{\ell\cap\R}$ of $\ell\cap \R$.
Furthermore, for any bounded interval $J\subset\R$, we have $\varepsilon^{-h}\int_J {\bf 1}_{E_{z,x,\varepsilon}}dx\to \nu_{\ell\cap\R}(J)$ for $\Theta_z$-a.s.~$\ell$ as well as in $L^2$. Therefore, for any compactly supported and continuous function $f$, we have
\[
\varepsilon^{-h}\int{\bf 1}_{E_{z,x,\varepsilon}}f(x,\ell)dx\Theta_z(d\ell)\to\int f(x,\ell)\nu_{\ell\cap\R}(dx)\Theta_z(d\ell).
\]
Since $\int f(x,\ell)\mathbbm{M}_\varepsilon(dx,d\ell)=
\frac{\varepsilon^h}{u(\varepsilon)}\varepsilon^{-h}\int{\bf 1}_{E_{z,x,\varepsilon}}f(x,\ell)dx\Theta_z(d\ell)$, which converges to $\int f(x,\ell)\mathbbm{M}(dx,d\ell)$ by~\eqref{eq:m-eps-1}, this implies the existence of the limit $\lim_{\varepsilon\to0}\frac{\varepsilon^h}{u(\varepsilon)}=C\in(0,\infty)$. Consequently,
\[
\int f(x,\ell)\mathbbm{M}(dx,d\ell)=C\int f(x,\ell)\nu_{\ell\cap\R}(dx)\Theta_z(d\ell),
\]
which gives the vague convergence
\begin{equation}\label{eq:m-eps-2}
\mathbbm{M}_\varepsilon\to C{\bf 1}_{\ell\cap\R\neq\emptyset}\nu_{\ell\cap\R}(dx)\Theta_z(d\ell)
\end{equation}
as $\varepsilon\to0$. Comparing~\eqref{eq:m-eps-1} and~\eqref{eq:m-eps-2}, by deweighting the total masses of the boundary Minkowski content measures, we obtain ${\bf 1}_{\gamma\text{ surrounds } z}\mu^\bub(d\gamma)=C{\bf 1}_{\ell\cap\R\neq\emptyset}\Theta_z(d\ell)$.

Now, since $\rho$ is obtained from the counting measure on boundary-touching loops and every $z\in\hH$ is surrounded by at most one CLE loop, we have ${\bf 1}_{\ell\text{ surrounds } z}\rho(d\ell)dz={\bf 1}_{\ell\cap\R\neq\emptyset}\Theta_z(d\ell)dz$, where $dz$ is the Lebesgue measure on $\hH$. Combined with the above result, we find ${\bf 1}_{\gamma\text{ surrounds } z}\mu^\bub(d\gamma)dz=C{\bf 1}_{\ell\text{ surrounds } z}\rho(d\ell)dz$. The result then follows by deweighting the Lebesgue areas of the regions surrounded by the bubble (resp.~loop) on both sides.
\end{proof}

According to Lemma~\ref{lem:mink}, the proof of Proposition~\ref{prop:touching-equals-bubble} also works when we consider $T_\varepsilon$ instead of $\wh T_\varepsilon$ (and consider the event that $\ell_z$ intersects $B(x,\varepsilon)$ in the definition of $E_{z,x,\varepsilon}$).
As a byproduct of the proof of Proposition~\ref{prop:touching-equals-bubble}, we record the following corollary.
\begin{corollary}\label{cor:sharp}
Let $T_\varepsilon$ and $\wh T_\varepsilon$ be as in Lemma~\ref{lem:bubble-cle}. Then both $\lim_{\varepsilon\to0}\varepsilon^{-h}\P[T_\varepsilon]$ and $\lim_{\varepsilon\to0}\varepsilon^{-h}\P[\wh T_\varepsilon]$ exist and are in $(0,\infty)$. In particular, the law of ${\bf 1}_{T_\varepsilon}\ell$ (or ${\bf 1}_{\wh T_\varepsilon}\ell$), times $\varepsilon^{-h}$, converges weakly to $C\mu_0^\bub$ restricted to surround $z$ for some $C\in(0,\infty)$.
\end{corollary}

\subsection{CLE boundary Green's functions}\label{sec:boundary-gf}

The aim of this section is to relate the $\CLE_\kappa$ boundary Green's functions defined in~\eqref{eq:four-type-gf} to the limit of boundary Green's function of chordal $\SLE_\kappa$; see Proposition~\ref{prop:gf-convergence-sec2}. Due to symmetry, it suffices to focus on $G^{(1234)}(x_1,x_2,x_3,x_4)$ and $G^{(12)(34)}(x_1,x_2,x_3,x_4)$. We refer readers to Appendix~\ref{appendix:gf} for further background of $\SLE_\kappa$ boundary Green's functions, including the various definitions appeared in the literature~\cite{lawler-mink-R,zhan-boundary-gf,fakhry2023} and their equivalence to the definition in this paper.

Based on Proposition~\ref{prop:touching-equals-bubble}, we can first express the $\CLE_\kappa$ boundary Green's functions using the $\SLE_\kappa$ bubble measure. Suppose $x,x_1,...,x_n\in\R$. Let $G^\bub_x(x_1,...,x_n)$ be the boundary Green's function of $\SLE_\kappa$ bubble rooted at $x$, defined by
\begin{equation}\label{eq:def-gf-bub}
    G^\bub_x(x_1,...,x_n)\prod_{i=1}^ndx_i=\mu^\bub_x\left[\prod_{i=1}^n\nu_{\gamma\cap\R}(dx_i)\right].
\end{equation}
Then $G^\bub_{x}(x_1,...,x_n)$ is finite and locally bounded; see Proposition~\ref{prop:gf-existence-bub}. Let $H_{\hH}(x_1,x_2)\propto\frac{1}{|x_1-x_2|^2}$ be the boundary Poisson kernel on $\hH$. For any simply connected domain $D$ and $a,b\in\partial D$, let $f$ be a conformal map from $D$ to $\hH$, and define $H_D(a,b)=|f'(a)||f'(b)|H_{\hH}(f(a),f(b))$ (when $\partial D$ is smooth near $a$ and $b$).
Proposition~\ref{prop:touching-equals-bubble} then implies the following
\begin{proposition}\label{prop:cle-gf-to-bubble}
Let ${\frak C}$ be the same constant as in Proposition~\ref{prop:touching-equals-bubble}. Then for any $n\ge2$, we have
\begin{equation}\label{eq:bub-1}
G^\bub_{x_1}(x_2,...,x_n)dx_1dx_2...dx_n={\frak C}\E\left[\sum_{\ell\in\cT(\Gamma)}\prod_{i=1}^n\nu_{\ell\cap\R}(dx_i)\right].
\end{equation}
When $n=2$, $G^\bub_{x_1}(x_2)={\frak C}H_{\hH}(x_1,x_2)^h$ (we choose the constant of $H_{\hH}(x_1,x_2)$ such that the coefficient is ${\frak C}$). In particular, for the $\CLE_\kappa$ boundary Green's functions defined in~\eqref{eq:four-type-gf}, we have
\begin{align}
G^{(1234)}(x_1,x_2,x_3,x_4)&= {\frak C}G^\bub_{x_1}(x_2,x_3,x_4)\in(0,\infty),\label{eq:bub-2}\\
G^{(12)(34)}(x_1,x_2,x_3,x_4)&={\frak C}\int H_{\cU(\gamma)}(x_3,x_4)^h\mu_{x_1,x_2}^\bub(d\gamma)\in(0,\infty)\label{eq:bub-3}
\end{align}
where $\mu_{x_1,x_2}^\bub$ is defined in~\eqref{eq:bubble-two-pinned}.
Recall that $\cU(A)$ denotes the unbounded connected component of $\hH\setminus A$ for a compact subset $A\subset\ol\hH$.
\end{proposition}
\begin{proof}
~\eqref{eq:bub-1} is a direct consequence of Proposition~\ref{prop:touching-equals-bubble} and~\eqref{eq:def-gf-bub}, and~\eqref{eq:bub-2} is the special case of~\eqref{eq:bub-1} with $n=4$. $G^\bub_{x_1}(x_2)\propto H_{\hH}(x_1,x_2)^h$ follows from the conformal covariance, and we choose the constant of the boundary Poisson kernel such that $G^\bub_{x_1}(x_2)={\frak C}H_{\hH}(x_1,x_2)^h$. This further implies
$H_{\cU(\gamma)}(x_3,x_4)^h={\frak C}\E_{\cU(\gamma)}\left[\sum_{\ell\in\cT(\Gamma)}\prod_{i=3}^4\nu_{\ell\cap\R}(dx_i)\right]$ (here $\E_{\cU(\gamma)}$ stands for taking expectations with respect to the $\CLE_\kappa$ on the domain $\cU(\gamma)$).
Then by~\eqref{eq:four-type-gf}, Proposition~\ref{prop:touching-equals-bubble} and the domain Markov property of $\CLE_\kappa$, we obtain~\eqref{eq:bub-3}.
\end{proof}

We now express the right sides of~\eqref{eq:bub-2} and~\eqref{eq:bub-3} as limits of boundary Green's function of chordal $\SLE_\kappa$. To this end, we start by defining two types of boundary Green's functions $F^{(1)}$ and $F^{(2)}$ of chordal $\SLE_\kappa$.
Let $y_1,y_2\in\R\cup\{\infty\}$, and let $\eta$ be a chordal $\SLE_\kappa$ on $\hH$ from $y_1$ to $y_2$ (whose law is denoted by $\mu_{\hH,y_1, y_2}^\#(d\eta)$). Let $\nu_{\eta\cap\R}(dx)$ be the $(1-h)$-dimensional Minkowski content of $\eta\cap\R$. For $x_1,...,x_n\in\R$, define the boundary $n$-point Green's function $G_{\hH,y_1,y_2}(x_1,...,x_n)$ of $\eta$ as
\begin{equation}\label{eq:gf}
G_{\hH,y_1,y_2}(x_1,...,x_n)\prod_{i=1}^n dx_i=\int \prod_{i=1}^n\nu_{\eta\cap\R}(dx_i)\mu_{\hH,y_1, y_2}^\#(d\eta),
\end{equation}
where the integration is taken over $\mu_{\hH,y_1, y_2}^\#(d\eta)$. \eqref{eq:gf} is equivalent to the boundary Green's function considered in~\cite{lawler-mink-R,fakhry2023}; see Proposition~\ref{prop:equivalence}.
\begin{definition}\label{def:F}
Suppose $x_1<x_2<x_3$ and $y_1,y_2\in\R\cup\{\infty\}$. Define $F^{(1)}, F^{(2)}$ as follows.
\begin{itemize}
    \item Let $F^{(1)}(y_1,y_2,x_1,x_2,x_3):=G_{\hH,y_1,y_2}(x_1,x_2,x_3)$.
    \item Denote $\rho_{\hH,y_1,y_2,x_1}^\#(d\eta)$ to be the law of chordal $\SLE_\kappa(\kappa-8)$ on $\hH$ from $y_1$ to $y_2$ with force point at $x_1$, which can be viewed as a chordal $\SLE_\kappa$ from $y_1$ to $y_2$ and conditioned to hit $x_1$. Let
    \begin{equation}\label{eq:gf-2}
    F^{(2)}(y_1,y_2,x_1,x_2,x_3):=G_{\hH,y_1,y_2}(x_1)\int H_{\cU(\eta)}(x_2,x_3)^{h}\rho_{\hH,y_1,y_2,x_1}^\#(d\eta).
    \end{equation}
    Here $\cU(\eta)$ is the unbouned connected component of $\hH\setminus\eta$, and $H_{\cU(\eta)}$ stands for the boundary Poisson kernel on $\cU(\eta)$. When $x_2$ or $x_3$ is not in $\ol {\cU(\eta)}$, we set $H_{\cU(\eta)}(x_2,x_3):=0$.
\end{itemize}
\end{definition}

The following lemma gives basic properties of $F^{(1)}$ and $F^{(2)}$.
\begin{lemma}\label{lem:f-basic}
For $j=1,2$, the functions $F^{(j)}(y_1,y_2,x_1,x_2,x_3)$ are finite and locally bounded. Furthermore, for $\eta$ sampled from $\mu_{\hH,y_1, y_2}^\#$ and parameterized by its half-plane capacity, let $g_t:\hH\setminus\eta_t\to\hH$ be the corresponding Loewner map such that $g_t(\eta_t)=W_t$. Then
\begin{equation}\label{eq:martingale}
M_t^{(j)}(y_1,y_2,x_1,x_2,x_3):=\left(\prod_{i=1}^3|g_t'(x_i)|^h\right)\cdot F^{(j)}(W_t,g_t(y_2),g_t(x_1),g_t(x_2),g_t(x_3))
\end{equation}
is a continuous local martingale.
\end{lemma}
\begin{proof}
For $F^{(1)}$, its finiteness and local boundedness follows from~\cite[Theorem 1]{fakhry2023} (see Proposition~\ref{prop:equivalence} for the equivalence of~\eqref{eq:gf} and the Green's function defined in~\cite{fakhry2023}). Given this, the finiteness and local boundedness of $F^{(2)}$ readily follows from~\eqref{eq:gf-2} and the monotonicity of boundary Poisson kernel, i.e.~$H_{\cU_\eta}(x_2,x_3)\le H_{\hH}(x_2,x_3)$. The local martingale property is the direct consequence of the domain Markov property of the chordal $\SLE_\kappa$ and conformal covariance.
\end{proof}

The following proposition gives that $G^{(1234)}(x_1,x_2,x_3,x_4)$ and $G^{(12)(34)}(x_1,x_2,x_3,x_4)$ can be obtained as normalized limits of $F^{(1)}$ and $F^{(2)}$.
\begin{proposition}\label{prop:gf-convergence-sec2}
Let $u<x_1<x_2<x_3$. Then as $y_1,y_2\to u$, we have
\begin{align*}
{\frak C}\lim_{y_1,y_2\to u}|y_2-y_1|^{-h}F^{(1)}(y_1,y_2;x_1,x_2,x_3)&=G^{(1234)}(u,x_1,x_2,x_3)\\
{\frak C}\lim_{y_1,y_2\to u}|y_2-y_1|^{-h}F^{(2)}(y_1,y_2;x_1,x_2,x_3)&=G^{(12)(34)}(u,x_1,x_2,x_3).
\end{align*}
where ${\frak C}$ is the same constant in Propositions~\ref{prop:touching-equals-bubble} and~\ref{prop:cle-gf-to-bubble}.
\end{proposition}
\begin{proof}
By Proposition~\ref{prop:gf-convergence}, we have $|y_2-y_1|^{-h}G_{\hH,y_1,y_2}(x_1,...,x_n)\to G_u^\bub(x_1,...,x_n)$. Combined with~\eqref{eq:bub-2} in Proposition~\ref{prop:cle-gf-to-bubble} and taking $n=3$ gives the first equation. For the second equation, note that as $y_1,y_2\to u$, the weak limit of $\rho_{\hH,y_1,y_2,x_1}^\#$ in~\eqref{eq:gf-2} is the $\SLE_\kappa$ bubble measure $\mu_{u,x_1}^\bub$ rooted at $u$ and $x_1$, normalized to be a probability measure $(\mu_{u,x_1}^\bub)^\#$. Thus, we have
\begin{align*}
{\frak C}\lim_{y_1,y_2\to u}|y_2-y_1|^{-h}F^{(2)}(y_1,y_2,x_1,x_2,x_3)&={\frak C}G_u^\bub(x_1)\int H_{\cU(\gamma)}(x_2,x_3)^{h}(\mu_{u,x_1}^\bub)^\#(d\gamma)\\
&={\frak C}\int H_{\cU(\gamma)}(x_2,x_3)^{h}\mu_{u,x_1}^\bub(d\gamma),
\end{align*}
which equals $G^{(12)(34)}(u,x_1,x_2,x_3)$ due to~\eqref{eq:bub-3} in Proposition~\ref{prop:cle-gf-to-bubble}.
\end{proof}

\subsection{Smoothness and second-order PDEs}\label{sec:smoothness}

Based on the local martingale property~\eqref{eq:martingale} in Lemma~\ref{lem:f-basic}, we are able to derive a second-order PDE satisfied by $F^{(j)}$ for $j=1,2$. However, in order to apply Ito's formula, we need to \emph{a priori} know the smoothness of $F^{(j)}$. Here we use H\"ormander's hypoellipticity to obtain the smoothness, which is inspired by the proof of~\cite[Theorem 6]{dub15localization} and~\cite[Remark 4.3]{fakhry2023}.
\begin{lemma}\label{lem:smooth}
For $j=1,2$, the functions $F^{(j)}(y_1,y_2,x_1,x_2,x_3)$ in Definition~\ref{def:F} are smooth on $\{(y_1,y_2,x_1,x_2,x_3)\in\R^5:x_1<x_2<x_3 \text{ and }y_1,y_2<x_1\}$.
\end{lemma}
\begin{proof}
Let $F^{(j)}(x_1,x_2,x_3)=F^{(j)}(0,\infty,x_1,x_2,x_3)$ for simplicity. By conformal covariance, it suffices to show the smoothness of $F^{(j)}(x_1,x_2,x_3)$ on $\{(x_1,x_2,x_3)\in\R_+^3:x_1<x_2<x_3\}$. For $1\le i\le3$, let $I_i$ be an open interval containing $x_i$, and $U=I_1\times I_2\times I_3\subset\R_+^3$. Define a second-order differential operator $\cL$ on $C^\infty(U)\cap C(\ol U)$ by
\begin{equation}\label{eq:generator}
\cL=\frac{\kappa}{2}{\bf X}^2+2{\bf Y}-2h \sum_{i=1}^3\frac{1}{x_i^2},\quad {\bf X}=\sum_{i=1}^3\partial_{x_i},\quad {\bf Y}=\sum_{i=1}^3\frac{1}{x_i}\partial_{x_i}.
\end{equation}
Note that the Lie brackets $[{\bf X},{\bf Y}]=-\sum_{i=1}^3\frac{1}{x_i^2}\partial_{x_i}$ and $[{\bf X},[{\bf X},{\bf Y}]]=2\sum_{i=1}^3\frac{1}{x_i^3}\partial_{x_i}$. Since $x_i$'s are mutually different on $U$, the Lie algebra generated by ${\bf X}$ and ${\bf Y}$ has rank $3$ on $U$. Thus, $\cL$ satisfies the H\"ormander’s condition\footnote{i.e.~for two vector fields ${\bf X}$ and ${\bf Y}$, ${\bf X},{\bf Y}$ and their iterated Lie brackets $[{\bf X},{\bf Y}], [{\bf X},[{\bf X},{\bf Y}]]$ etc.~span the whole tangent space at every point.}. By~\cite[Theorem 5.2]{hormander}, we can then define a Poisson operator $\cP:C(\partial U)\to C^\infty(U)\cap C(\ol U)$ for $\cL$ such that for any $\omega\in C(\partial U)$, $(\cP\omega)|_{\partial U}=\omega$ and $\cL(\cP\omega)=0$.

Let $\varphi_t$ be the Loewner map with driving function $W_t=\sqrt{\kappa}B_t$. By~\eqref{eq:martingale}, we know $M_t^{(j)}(x_1,x_2,x_3):=\prod_{i=1}^3|\varphi'(x_i)|^h F^{(j)}(x_1,x_2,x_3)$ is a continuous local martingale. For each $t\ge0$, let $X_t:=(g_t(x_1)-W_t,g_t(x_2)-W_t,g_t(x_3)-W_t)$, and let $\tau$ be the first hitting time of $\partial U$ for $(X_t)$.
By optional stopping theorem, $M_0^{(j)}(X_0)=\E[M_\tau^{(j)}(X_0)]$.
On the other hand, note that $\cL$ defined in~\eqref{eq:generator} is the infinitesimal generator of $(X_t)$ with the killing rate $\frac{2h}{X_t^2}$.
Let $N^{(j)}:=\cP (M_0^{(j)}|_{\partial U})\in C^\infty(U)$ (thus $\cL N^{(j)}=0$). Applying Dynkin's formula to $(X_t)$ yields
\[
N^{(j)}(X_0)=\E\left[e^{-\int_0^\tau\frac{2h}{X_t^2}dt}N^{(j)}(X_\tau)\right]=\E[|\varphi_\tau'(X_0)|^h N^{(j)}(X_\tau)].
\]
Since $X_\tau\in\partial U$, we also have
\[
\E[|\varphi_\tau'(X_0)|^h N^{(j)}(X_\tau)]=\E[|\varphi_\tau'(X_0)|^h M_0^{(j)}(X_\tau)]=\E[M_\tau^{(j)}(X_0)]=M_0^{(j)}(X_0)=F^{(j)}(X_0).
\]
Therefore, $F^{(j)}=N_0^{(j)}\in C^\infty(U)$. The result then follows from varying $U$.
\end{proof}

Now we are able to derive the second-order PDEs satisfied by $F^{(j)}$.

\begin{proposition}\label{prop:pde-u-v}
Let $u=\frac{1}{2}(y_1+y_2)$ and $v=\frac{1}{2}(y_2-y_1)$. Then for $j=1,2$, $F^{(j)}(y_1,y_2,x_1,x_2,x_3)$ satisfies the following pair of second-order PDEs:
\begin{equation}\label{eq:pde}
\begin{aligned}
\left(\frac{\kappa}{4}(\partial_{uu}+\partial_{vv})+\frac{1+\frac{\kappa-6}{2}}{v}\partial_v+\sum_{i=1}^3\left(\frac{4(x_i-u)\partial_{x_i}}{(x_i-u)^2-v^2}-\frac{2h}{(x_i-u+v)^2}-\frac{2h}{(x_i-u-v)^2}\right)\right)F=0,\\
\left(\frac{\kappa}{2}\partial_{uv}-\frac{1-\frac{\kappa-6}{2}}{v}\partial_u+\sum_{i=1}^3\left(\frac{4v\partial_{x_i}}{(x_i-u)^2-v^2}+\frac{2h}{(x_i-u+v)^2}-\frac{2h}{(x_i-u-v)^2}\right)\right)F=0.
\end{aligned}
\end{equation}
\end{proposition}

\begin{proof}
Let $\eta$ be parameterized by its half-plane capacity, and let $g_t:\hH\setminus\eta_t\to\hH$ be the corresponding Loewner map. According to SLE coordinate change~\cite{sw-coord}, $(\eta_t)$ can be viewed as a chordal $\SLE_\kappa(\kappa-6)$ on $\hH$ from $y_1$ to $\infty$, with force point at $y_2$. Thus, we have
\begin{align*}
\partial_t g_t(z)=\frac{2}{g_t(z)-W_t},\quad dW_t=\sqrt{\kappa}dB_t-\frac{\kappa-6}{g_t(y_2)-W_t}dt.
\end{align*}
By Lemma~\ref{lem:f-basic}, $M_t^{(j)}:=\left(\prod_{i=1}^3|g_t'(x_i)|^h\right)\cdot F^{(j)}(W_t,g_t(y_2),g_t(x_1),g_t(x_2),g_t(x_3))$ is a local martingale. By Lemma~\ref{lem:smooth}, applying Ito's formula to $(M_t^{(j)})$ gives
\begin{align*}
\left(-\sum_{i=1}^3\frac{2h}{(x_i-y_1)^2}-\frac{\kappa-6}{y_2-y_1}\partial_{y_1}+\frac{2}{y_2-y_1}\partial_{y_2}+\sum_{i=1}^3\frac{2}{x_i-y_1}\partial_{x_i}+\frac{\kappa}{2}\partial_{y_1}^2\right)F^{(j)}=0.
\end{align*}
This is also known as the second-order BPZ equation at $y_1$.
Let $u=\frac{1}{2}(y_1+y_2)$ and $v=\frac{1}{2}(y_2-y_1)$. Then equivalently,
\begin{align}\label{eq:pde-plus}
\left(-\sum_{i=1}^3\frac{2h}{(x_i-u+v)^2}+\frac{1}{2v}\left((\partial_u+\partial_v)-\frac{\kappa-6}{2}(\partial_u-\partial_v)\right)+\sum_{i=1}^3\frac{2}{x_i-u+v}\partial_{x_i}+\frac{\kappa}{8}(\partial_u-\partial_v)^2\right)F^{(j)}=0.
\end{align}
By reversibility, we can also view $\eta$ as a chordal $\SLE_\kappa$ from $y_2$ to $y_1$ (hence is equivalent to a chordal $\SLE_\kappa(\kappa-6)$ from $y_2$ to $\infty$ with force point at $y_1$). Then similarly, ~\eqref{eq:pde-plus} also holds with $v$ replaced by $-v$. Namely, we have
\begin{align}\label{eq:pde-minus}
\left(-\sum_{i=1}^3\frac{2h}{(x_i-u-v)^2}-\frac{1}{2v}\left((\partial_u-\partial_v)-\frac{\kappa-6}{2}(\partial_u+\partial_v)\right)+\sum_{i=1}^3\frac{2}{x_i-u-v}\partial_{x_i}+\frac{\kappa}{8}(\partial_u+\partial_v)^2\right)F^{(j)}=0.
\end{align}
Combining~\eqref{eq:pde-plus} and~\eqref{eq:pde-minus} gives~\eqref{eq:pde}.
\end{proof}

\section{Fusion}\label{sec:fusion}
For $j=1,2$, recall that Proposition~\ref{prop:gf-convergence-sec2} establishes the convergence of $v^{-h} F^{(j)}(y_1,y_2,x_1,x_2,x_3)$ to $G^{(1234)}(u,x_1,x_2,x_3)$ (or $G^{(12)(34)}(u,x_1,x_2,x_3)$), as $v:=\frac12(y_2-y_1)\to0$ with $u:=\frac12(y_1+y_2)$ fixed. Moreover, Proposition~\ref{prop:pde-u-v} gives the two second-order PDEs satisfied by $F^{(j)}(y_1,y_2,x_1,x_2,x_3)$. 
Following the framework of~\cite{dubedat2015}, these ingredients will together imply that the limiting function $G^{(1234)}(u, x_1, x_2, x_3)$ (or $G^{(12)(34)}(u, x_1, x_2, x_3)$) satisfies a third-order differential equation. The aim of the current section is to derive this third-order equation, thereby proving Theorem~\ref{thm:ode}.

We will rely on the following input from~\cite{dubedat2015}.
\begin{lemma}[{\cite[Lemma 2]{dubedat2015}}]\label{lem:dub}
Let $n\ge1$, $\bx=(x_1,...,x_n)$ and $\varepsilon>0$. Let $U=\{(y_1,y_2,\bx):|y_1-u|<\varepsilon, 0\le y_2-y_1<\varepsilon,\|\bx-{\bf x}_0\|<\varepsilon\}$, and $\Delta=\{(y_1,y_2,\bx)\in U:y_1=y_2\}$. For $\rho,\tau,\sigma\in\R$, consider the differential operator
\begin{equation}\label{eq:operator-M}
\cM=\frac{1}{2}\partial_{y_1}^2+\left(\frac{\rho}{y_1-y_2}+a(y_1,y_2,\bx)\right)\partial_{y_1}+\left(\frac{\tau}{2(y_2-y_1)}+b(y_1,y_2,\bx)\right)\partial_{y_2}+{\bf X}+\left(-\frac{\tau \sigma}{2(y_1-y_2)^2}+\frac{d(y_1,y_2,\bx)}{y_2-y_1}\right)
\end{equation}
where ${\bf X}:=\sum_{i=1}^nc_i(y_1,\bx)\partial_{x_i}$, and $a,b,c,d$ are smooth on $U$. Suppose $\alpha_-<\alpha_+$ are the two roots of the indicial equation $\alpha(\alpha-1)+(\tau+2\rho)\alpha-\tau \sigma=0$.

Suppose $f$ is a real-valued smooth function on $U\setminus\Delta$ such that $\cM f=0$ and $f=O((y_2-y_1)^{\alpha_-+\delta})$ for some $\delta>0$. Further, assume that for ${\bf Y}:=\partial_{y_1}$ and ${\bf Y},{\bf X}$ satisfies the H\"ormander's condition on $U\setminus\Delta$. Then there exists a smooth function $g$ on $U$ such that $f=(y_2-y_1)^{\alpha_+}g$.
\end{lemma}

We now prove Theorem~\ref{thm:ode} based on Lemma~\ref{lem:dub}, following the approach as explained in~\cite[Section 2]{dubedat2015}. By symmetry, we will focus on the case $u<x_1<x_2<x_3$.

\begin{proof}[Proof of Theorem~\ref{thm:ode}]
In~\eqref{eq:operator-M}, we take $n=3$, $\rho=\frac{\kappa-6}{\kappa}$, $\tau=\frac{4}{\kappa}$, $\sigma=0$, $a=b=0$, $c_i=\frac{2/\kappa}{x_i-y_1}$ and $d=-\sum_{i=1}^3\frac{2h/\kappa}{(x_i-y_1)^2}(y_2-y_1)$ (hence $\alpha_-=0$, $\alpha_+=h$). Then for $j=1,2$, $F^{(j)}$ is smooth such that $\cM F^{(j)}=0$. Furthermore, by Proposition~\ref{prop:gf-convergence-sec2}, $F^{(j)}=O((y_2-y_1)^h)$ as $y_2\to y_1$. By direct computation, we have $[{\bf Y},{\bf X}]=\sum_{i=1}^3\frac{2/\kappa}{(x_i-y_1)^2}\partial_{x_i}$ and $[{\bf Y},[{\bf Y},{\bf X}]]=\sum_{i=1}^3\frac{4/\kappa}{(x_i-y_1)^3}\partial_{x_i}$, thus the H\"ormander's condition holds. Hence, we can apply Lemma~\ref{lem:dub} to $F^{(j)}$ to see that there exists a smooth $g^{(j)}$ on $U$ such that $F^{(j)}=v^h g^{(j)}(u,v,x_1,x_2,x_3)$.

Consider the Taylor expansion for $g^{(j)}=\sum_{n\ge1}^Ng_{n}^{(j)}v^n+O(v^{N+1})$ near $v=0$, where each $g_{n}^{(j)}$ is smooth. Since $F^{(j)}$ is invariant under $v\leftrightarrow-v$, only even powers appear, i.e.~$g_{2n+1}^{(j)}=0$ for every $n\ge0$. Consequently, $F^{(j)}$ has the expansion
\begin{equation}\label{eq:ansatz-1}
\begin{aligned}
F^{(j)}=v^h\sum_{n=0}^Nv^{2n}g_{2n}^{(j)}(u,x_1,x_2,x_3)+O(v^{2N+2})
\end{aligned}
\end{equation}
as $v\to0$. In particular, by Proposition~\ref{prop:gf-convergence-sec2}, we have
\begin{equation}\label{eq:g0-G}
g_0^{(1)}(u,x_1,x_2,x_3)=CG^{(1234)}(u,x_1,x_2,x_3),\quad g_0^{(2)}(u,x_1,x_2,x_3)=CG^{(12)(34)}(u,x_1,x_2,x_3)
\end{equation}
for some constant $C\in(0,\infty)$.
The expansion for partial derivatives of $F^{(j)}$ is similar, giving that
\begin{equation}\label{eq:ansatz-2}
\begin{aligned}
\partial_uF^{(j)}&=v^h\sum_{n=0}^Nv^{2n}\partial_ug_{2n}^{(j)}(u,x_1,x_2,x_3)+O(v^{2N+2}),
\\\partial_vF^{(j)}&=v^h\sum_{n=0}^N(2n+h)v^{2n-1}g_{2n}^{(j)}(u,x_1,x_2,x_3)+O(v^{2N+1})
\end{aligned}
\end{equation}
etc. (Note that $h\in(0,1)$ is not an integer).
Taking the above expansions~\eqref{eq:ansatz-1},~\eqref{eq:ansatz-2} into~\eqref{eq:pde}, by comparing the coefficients of $v^{h-2}$ and $v^{h-1}$, we first obtain (the ``zeroth order'' equations)
\begin{align*}
\frac{\kappa}{4}h(h-1)+\left(1+\frac{\kappa-6}{2}\right)h=0,\\
\frac{\kappa}{2}h-\left(1-\frac{\kappa-6}{2}\right)=0,
\end{align*}
which both hold since $h=\frac{8}{\kappa}-1$. Iteratively, by comparing the coefficients of $v^h$ and $v^{h+1}$, we have (the ``first order'' equations)
\begin{align*}
\frac{\kappa}{4}\left(\partial_{uu} g_0^{(j)}+(h+2)(h+1)g_2^{(j)}\right)+\left(1+\frac{\kappa-6}{2}\right)(h+2)g_2^{(j)}+\sum_{i=1}^3\left(\frac{4\partial_{x_i}g_0^{(j)}}{x_i-u}-\frac{4hg_0^{(j)}}{(x_i-u)^2}\right)=0,\\
\frac{\kappa}{2}(h+2)\partial_ug_2^{(j)}-\left(1-\frac{\kappa-6}{2}\right)\partial_ug_2^{(j)}+\sum_{i=1}^3\left(\frac{4\partial_{x_i}g_0^{(j)}}{(x_i-u)^2}-\frac{8hg_0^{(j)}}{(x_i-u)^3}\right)=0.
\end{align*}
Combining the above two equations to eliminate $g_2^{(j)}$ and taking into $h=\frac{8}{\kappa}-1$, we obtain
\begin{equation}\label{eq:fusion-pde}
\frac{\kappa}{4}\partial_u^3 g_0^{(j)}+\frac{1}{2}\left(1-\frac{8}{\kappa}\right)\sum_{i=1}^3\left(\frac{4\partial_{x_i}g_0^{(j)}}{(x_i-u)^2}-\frac{8hg_0^{(j)}}{(x_i-u)^3}\right)+\sum_{i=1}^3\left(\frac{4\partial_{ux_i}g_0^{(j)}}{x_i-u}-\frac{4h\partial_ug_0^{(j)}}{(x_i-u)^2}\right)=0.
\end{equation}
Recall that by~\eqref{eq:g0-G} and~\eqref{eq:u-lambda}, for the cross-ratio $\lambda=\frac{(x_1-u)(x_3-x_2)}{(x_3-x_1)(x_2-u)}$, we have
\begin{equation}\label{eq:g0-U}
g_0^{(1)}(u,x_1,x_2,x_3)=\frac{(1-\lambda)^{-2h}U^{(1234)}(\lambda)}{((x_1-u)(x_3-x_2))^{2h}},\quad g_0^{(2)}(u,x_1,x_2,x_3)=\frac{(1-\lambda)^{-2h}U^{(12)(34)}(\lambda)}{((x_1-u)(x_3-x_2))^{2h}}
\end{equation}
Substituting~\eqref{eq:g0-U} into~\eqref{eq:fusion-pde}, (after a long but standard calculation; see Appendix~\ref{appendix:matlab} for a verification using MATLAB) we obtain that $U^p(\lambda)$ satisfies~\eqref{eq:ode} for $p\in\{(1234),(12)(34)\}$. Finally, since $U^{(14)(23)}(\lambda)=U^{(12)(34)}(1-\lambda)$ and~\eqref{eq:ode} is invariant under $\lambda\leftrightarrow1-\lambda$, we have $U^{(14)(23)}(\lambda)$ also satisfies~\eqref{eq:ode}, so is $U^{\rm total}(\lambda)$.
\end{proof}

\section{Identification of solutions}\label{sec:solution}

Recall that $\kappa\in(4,8)$ and $h=\frac{8}{\kappa}-1\in(0,1)$.
Theorem~\ref{thm:identify} is based on the following
\begin{proposition}\label{prop:identify-v0}
As $x_2\to x_1$ (with $x_1,x_3,x_4$ fixed), we have
\[
G^{(1234)}(x_1,x_2,x_3,x_4)+G^{(12)(34)}(x_1,x_2,x_3,x_4)=H_{\hH}(x_1,x_2)^hH_{\hH}(x_3,x_4)^h(1+o(|x_2-x_1|^h)).
\]
\end{proposition}

\begin{proof}[Proof of Theorem~\ref{thm:identify}, assuming Proposition~\ref{prop:identify-v0}]
Let $x_1<x_2<x_3<x_4$. The identification of $U^{(14)(23)}(\lambda)$ and $U^{(12)(34)}(\lambda)$ is straightforward. Indeed, by~\eqref{eq:bub-3} in Proposition~\ref{prop:cle-gf-to-bubble}, we have
\begin{align*}
G^{(14)(23)}(x_1,x_2,x_3,x_4)\propto\int H_{\cU(\gamma)}(x_2,x_3)^h \mu_{x_4,x_1}^\bub(d\gamma)&\le H_{\hH}(x_1,x_4)^h \int H_{\cU(\eta)}(x_2,x_3)^h\mu_{\hH,x_4,x_1}^\#(d\eta)\\
&=O(|x_2-x_1|^{\frac{4}{\kappa}}),\quad \text{as }  x_2\to x_1 \text{ and }x_1,x_3,x_4 \text{ fixed}.
\end{align*}
The inequality follows from that $H_{\cU(\gamma)}(x_2,x_3)$ under the law of $(\mu_{x_4,x_1}^\bub)^\#(d\gamma)$ is stochastically dominated by $H_{\cU(\eta)}(x_2,x_3)$ under $\mu_{\hH,x_4,x_1}^\#(d\eta)$, according to the decomposition of two-point rooted $\SLE_\kappa$ bubbles. The final equality follows from~\cite[Lemma 3.4, Proposition 3.5]{wu-hsle} (with taking $\nu=2$ there; see also Lemma~\ref{lem:hsle} below). Therefore, $U^{(14)(23)}(\lambda)=O(\lambda^{\frac{12}{\kappa}-1})=o(\lambda^h)$ as $\lambda\to0$, and hence $U^{(14)(23)}(\lambda)=C_1V_{3h+1}(\lambda)$ for some $C_1\in(0,\infty)$. The identification of $U^{(12)(34)}$ then follows from symmetry  $U^{(12)(34)}(\lambda)= U^{(14)(23)}(1-\lambda)$.

The identification of $U^{\rm total}(\lambda)$ relies on Proposition~\ref{prop:identify-v0}.
Since $U^{\rm total}(\lambda)$ is the solution of~\eqref{eq:ode}, thus there exists $C_2,\alpha,\beta$ such that $U^{\rm total}(\lambda)=C_2(V_0(\lambda)+\alpha V_h(\lambda)+\beta V_{3h+1}(\lambda))$. By Proposition~\ref{prop:identify-v0}, $\alpha=0$. On the other hand, since $U^{\rm total}(\lambda)=U^{\rm total}(1-\lambda)$,  Proposition~\ref{prop:identify-v0} also gives $U^{\rm total}(1-\lambda)=C_2(1+o(\lambda^h))$ as $\lambda\to0$. Note that $U^{(1234)}(\lambda)=O(\lambda^h)$ as $\lambda\to0$ (this can be seen e.g.~by combining Proposition~\ref{prop:gf-convergence-sec2} and the asymptotic behavior of the boundary Green's function of chordal $\SLE_\kappa$~\cite[Theorem 1.1]{fakhry2023}). By Proposition~\ref{prop:identify-v0}, we must have $U^{(12)(34)}(\lambda)=C_2+O(\lambda^h)$ and hence $V_{3h+1}(1-\lambda)=\frac{C_2}{C_1}+O(\lambda^h)$ as $\lambda\to0$. Thus such $\beta$ is unique. The ratio $\frac{C_1}{C_2}$ is determined by $\frac{U^{(12)(34)}(\lambda)}{U^{\rm total}(\lambda)}\to1$ as $\lambda\to0$.
\end{proof}

In this section, we first prove Proposition~\ref{prop:identify-v0} in Section~\ref{sec:identify-v0}, which relies on expressing $G^{(1234)}$ and $G^{(12)(34)}$ in terms of the partition function of $\CLE_\kappa$ with two wired boundary arcs~\cite{MW18} (see Corollary~\ref{cor:12-34}), and the explicit subleading behavior of the latter~\eqref{eq:z-tau}. Then in Section~\ref{sec:kappa=6}, we prove Theorems~\ref{thm:percolation} and~\ref{thm:fk} as applications of Theorems~\ref{thm:ode} and~\ref{thm:identify}. As we mentioned before, in these cases, the explicit forms of $V_0,V_h,V_{3h+1}$ were previously obtained in~\cite{GV,Gori:2018gqx}. We discuss solutions of~\eqref{eq:ode} for other special $\kappa$'s in Section~\ref{sec:kappa-le-4}.

\subsection{Proof of Proposition~\ref{prop:identify-v0}}\label{sec:identify-v0}
In this section we prove Proposition~\ref{prop:identify-v0}. We fix
\[
b=\frac{6-\kappa}{2\kappa}
\]
throughout this section.
Suppose $D$ is a Jordan domain, and $x,y\in\partial D$. Let $\mu_{D,x,y}^\#$ be the law of chordal $\SLE_\kappa$ on $D$ from $x$ to $y$, and $\wt\mu_{D,x,y}^\#$ be the law of chordal $\SLE_\kappa(2)$ on $D$ from $x$ to $y$, with the force point $x-$. When $\partial D$ is smooth near $x$ and $y$, we denote $\mu_{D,x,y}=H_D(x,y)^b\mu_{D,x,y}^\#$ and $\wt\mu_{D,x,y}=H_D(x,y)^h\wt\mu_{D,x,y}^\#$. Note that $\wt\mu_{D,x,y}$ is different from $\wt\mu_{D,y,x}$, while their total masses coincide. Here we choose the constant of the boundary Poisson kernel as in Proposition~\ref{prop:cle-gf-to-bubble}.

We start by the following forms of $G^{(1234)}(x_1,x_2,x_3,x_4)$ and $G^{(12)(34)}(x_1,x_2,x_3,x_4)$. According to symmetry, we focus on the case
\[
x_1<x_2<x_3<x_4
\]
throughout this section. Let $G_{D,x_1,x_2}(x_3,x_4)$ be the boundary two-point Green's function of a chordal $\SLE_\kappa$ from $x_1$ to $x_2$ on $D$ at $(x_3,x_4)$ (when $\partial D$ is smooth near $x_3,x_4$) such that
\[
G_{D,x_1,x_2}(x_3,x_4)=|\phi'(x_3)\phi'(x_4)|^hG_{\hH,\phi(x_1),\phi(x_2)}(\phi(x_3),\phi(x_4))
\]
for a conformal map $\phi:D\to\hH$, where $G_{\hH,\phi(x_1),\phi(x_2)}(\phi(x_3),\phi(x_4))$ is defined in~\eqref{eq:gf}.
\begin{lemma}\label{lem:re-express}
We have
\begin{align}
G^{(1234)}(x_1,x_2,x_3,x_4)
&=\int G_{\cU(\eta_{12}),x_1,x_2}(x_4,x_3)\wt \mu_{\hH,x_1,x_2}(d\eta_{12}),\label{eq:re-express-1}\\
G^{(12)(34)}(x_1,x_2,x_3,x_4)&=\int H_{\cU(\eta')}(x_3,x_4)^h\mu^\#_{\cU(\eta_{12}),x_1,x_2}(d\eta')\wt\mu_{\hH,x_1,x_2}(d\eta_{12}).\label{eq:re-express-2}
\end{align}
\end{lemma}
\begin{proof}
\eqref{eq:re-express-2} follows from~\eqref{eq:bub-3} in Proposition~\ref{prop:cle-gf-to-bubble} and the decomposition of the $\SLE_\kappa$ bubble measure with two marked points (see Section~\ref{sec:sle-bubble}). For~\eqref{eq:re-express-1}, note that combining Proposition~\ref{prop:touching-equals-bubble},~\eqref{eq:bubble-two-pinned} and the decomposition above gives
\[
\E\left[\sum_{\ell\in\cT(\Gamma)}\prod_{i=1}^4\nu_{\ell\cap\R}(dx_i)\right]=\int \nu_{\eta_{12}\cap\R}(dx_3)\nu_{\eta_{12}\cap\R}(dx_4)\mu_{\cU(\eta_{12}),x_1,x_2}^\#(d\eta')\wt\mu_{\hH,x_1,x_2}(d\eta_{12}) dx_1dx_2.
\]
Here the integral on the right side is with respect to the measure $\wt\mu_{\hH,x_1,x_2}(d\eta_{12})$. The result then follows from~\eqref{eq:gf} and the conformal covariance of Minkowski content.
\end{proof}

We prove Proposition~\ref{prop:identify-v0} by a detailed analysis of the right sides of~\eqref{eq:re-express-1} and~\eqref{eq:re-express-2}. We record the following result from~\cite{wu-hsle}, which arises from the partition function of hypergeometric SLE\footnote{The term hypergeometric SLE was earlier introduced in~\cite{qian-trichordal} to refer to a broader class of SLEs.}.

\begin{lemma}[{\cite[Lemma 3.4, Proposition 3.5]{wu-hsle}}]\label{lem:hsle}
For $\eta_{12}$ sampled from $\mu_{\hH,x_1,x_2}(d\eta_{12})$, recall that $\cU(\eta_{12})$ is the unbounded connected component of $\hH\setminus\eta_{12}$. For $\nu\ge\frac{\kappa}{2}-4$, let $\alpha=\frac{\nu+2}{\kappa}$ and $\beta=\frac{(\nu+2)(\nu+6-\kappa)}{4\kappa}$. Then we have
\[
\int H_{\cU(\eta_{12})}(x_3,x_4)^\beta\mu_{\hH,x_1,x_2}(d\eta_{12})= H_{\hH}(x_1,x_2)^b H_{\hH}(x_3,x_4)^\beta \frac{(1-\lambda)^\alpha{}_2F_1(2\alpha,1-\frac{4}{\kappa};2\alpha+\frac{4}{\kappa};1-\lambda)}{{}_2F_1(2\alpha,1-\frac{4}{\kappa};2\alpha+\frac{4}{\kappa};1)},\]
where $\lambda=\frac{(x_2-x_1)(x_4-x_3)}{(x_3-x_1)(x_4-x_2)}$ is the cross-ratio.
\end{lemma}

We mainly use the $\nu=0$ case of Lemma~\ref{lem:hsle}, which also appeared earlier in~\cite{BBK-05,dubedat-euler}.
In the following, we write
\begin{equation}\label{eq:f(x)}
f(x)=x^{\frac{2}{\kappa}}(1-x)^{1-\frac{6}{\kappa}}\frac{{}_2F_1(\frac{4}{\kappa},1-\frac{4}{\kappa};\frac{8}{\kappa};x)}{{}_2F_1(\frac{4}{\kappa},1-\frac{4}{\kappa};\frac{8}{\kappa};1)}.
\end{equation}
We will rely on a key observation on the explicit subleading behavior of $f(x)$ in the proof of Proposition~\ref{prop:identify-v0}; see~\eqref{eq:z-tau} below.
The following lemma deals with the right side of~\eqref{eq:re-express-2}.

\begin{lemma}\label{lem:cascade}
Consider the measure $\wt\mu_{\cU(\eta_{12}),x_3,x_4}(d\eta_{34})\mu_{\hH,x_1,x_2}(d\eta_{12})$. Then its total mass
\begin{equation}\label{eq:cascade}
\int H_{\cU(\eta_{12})}(x_3,x_4)^h\mu_{\hH,x_1,x_2}(d\eta_{12})=\int H_{\cU(\eta)}(x_1,x_2)^b\tau ^{2b}f(1-\tau)\wt\mu_{\hH,x_3,x_4}(d\eta).
\end{equation}
Here $\tau\in(0,1)$ is such that the $(\cU(\eta),x_1,x_2,x_3,x_4)$ is conformally equivalent to $(\hH,0,\tau,1,\infty)$.
\end{lemma}
\begin{proof}
The left side of~\eqref{eq:cascade} equals the total mass of $\wt\mu_{\cU(\eta_{12}),x_4,x_3}(d\eta_{34})\mu_{\hH,x_1,x_2}(d\eta_{12})$. We first claim that it equals $\int H_{\cU(\eta')}(x_1,x_2)^b\wt\mu_{\hH,x_4,x_3}(d\eta')$ (i.e.~the commutation relation of bi-chordal $\SLE_\kappa$ and $\SLE_\kappa(2)$ pair). Define $M_{\hH,x_3,x_4,\varepsilon}:=\mu_{\cU(\eta'),x_3,x_4}(d\eta_{34})\mu_{\hH,x_3-\varepsilon,x_4+\varepsilon}(d\eta')$. By~\cite[Lemma 3.7]{wu-hsle}, as $\varepsilon\to0$, the marginal law of $\eta_{34}$ under $M_{\hH,x_3,x_4,\varepsilon}$ (which corresponds to the hypergeometric SLE with parameter $\nu=0$ there), times $\varepsilon^{-2h}$, converges weakly to $C\wt\mu_{\hH,x_4,x_3}(d\eta_{34})$ for some $C\in(0,\infty)$. Now consider the measure defined on the triples
$N_\varepsilon:=\mu_{\cU(\eta_{34}),x_1,x_2}(d\eta_{12})M_{\hH,x_3,x_4,\varepsilon}(d\eta_{34},d\eta')$. Taking $\varepsilon\to0$, the marginal law of $(\eta_{12},\eta_{34})$ under $N_\varepsilon$ thus converges weakly to $C\mu_{\cU(\eta_{34}),x_1,x_2}(d\eta_{12})\wt\mu_{\hH,x_4,x_3}(d\eta_{34})$. On the other hand, using the symmetry for bi-chordal $\SLE_\kappa$ (see e.g.~\cite[Proposition 6.10]{wu-hsle}) twice, we also have $N_\varepsilon=M_{\cU(\eta_{12}),x_3,x_4,\varepsilon}(d\eta_{34},d\eta')\mu_{\hH,x_1,x_2}(d\eta_{12})$. Then as $\varepsilon\to0$, $N_\varepsilon$ also converges weakly to $C\wt\mu_{\cU(\eta_{12}),x_4,x_3}(d\eta_{34})\mu_{\hH,x_1,x_2}(d\eta_{12})$. Consequently, we have
\[
\mu_{\cU(\eta_{34}),x_1,x_2}(d\eta_{12})\wt\mu_{\hH,x_4,x_3}(d\eta_{34})=\wt\mu_{\cU(\eta_{12}),x_4,x_3}(d\eta_{34})\mu_{\hH,x_1,x_2}(d\eta_{12}),
\]
and the claim follows by comparing the total masses on both sides.

Note that $\wt\mu_{\hH,x_4,x_3}(d\eta')$ can be obtained by first sampling $\eta$ from $\wt\mu_{\hH,x_3,x_4}(d\eta)$, and then sampling $\eta'$ from $\mu_{\cU(\eta),x_3,x_4}^\#$. Combined with the above claim, we have
\[
\int H_{\cU(\eta_{12})}(x_3,x_4)^h\mu_{\hH,x_1,x_2}(d\eta_{12})=\int H_{\cU(\eta')}(x_1,x_2)^b\mu_{\cU(\eta),x_3,x_4}^\#(d\eta') \wt\mu_{\hH,x_3,x_4}(d\eta).
\]
By Lemma~\ref{lem:hsle} with $\nu=0$ (hence $\alpha=\frac{2}{\kappa}$ and $\beta=b$), the right side above equals
\[
\int H_{\cU(\eta)}(x_1,x_2)^b(1-\tau)^{\frac{2}{\kappa}}\frac{{}_2F_1(\frac{4}{\kappa},1-\frac{4}{\kappa};\frac{8}{\kappa};1-\tau)}{{}_2F_1(\frac{4}{\kappa},1-\frac{4}{\kappa};\frac{8}{\kappa};1)}\wt\mu_{\hH,x_3,x_4}(d\eta)=\int H_{\cU(\eta)}(x_1,x_2)^b\tau ^{2b}f(1-\tau)\wt\mu_{\hH,x_3,x_4}(d\eta),
\]
as desired.
\end{proof}

Next, we use the connection probability of $\CLE_\kappa$ with two wired boundary arcs~\cite{MW18} to derive a similar expression for the boundary two-point Green's function of chordal $\SLE_\kappa$, which deals with the right side of~\eqref{eq:re-express-1}. We refer readers to~\cite[Section 2]{MW18} for backgrounds on CLE with two wired boundary arcs.
Recall the boundary two-point Green's function $G_{\hH,x_1,x_2}(x_3,x_4)$ of the chordal $\SLE_\kappa$ defined in~\eqref{eq:gf}.
\begin{lemma}\label{lem:cascade-2}
We have
\[
H_{\hH}(x_1,x_2)^b G_{\hH,x_1,x_2}(x_3,x_4) =\frac{1}{-2\cos\left(\frac{4\pi}{\kappa}\right)}\int H_{\cU(\eta)}(x_1,x_2)^b\tau ^{2b}f(\tau)\wt\mu_{\hH,x_3,x_4}(d\eta).
\]
Here $\tau$ is such that the $(\cU(\eta),x_1,x_2,x_3,x_4)$ is conformally equivalent to $(0,\tau,1,\infty)$.
\end{lemma}
\begin{proof}
Let $\eta_{12}$ be sampled from $\mu_{\hH,x_1,x_2}^\#$, and $\Gamma$ be sampled from an independent $\CLE_\kappa$ on each connected component of $\hH\setminus\eta_{12}$. Then $\wh\Gamma=\Gamma\cup\{\eta_{12}\}$ is a $\CLE_\kappa$ on $\hH$ with a wired boundary arc $[x_1,x_2]$. For $i=3,4$ and $r_i>0$, let $I(x_i,r_i)=(x_i-r_i,x_i+r_i)$. Denote $E_{r_3,r_4}$ to be the event that there exists an element in $\wh\Gamma$ that intersects both $I(x_3,r_3)$ and $I(x_4,r_4)$. On $E_{r_3,r_4}$, let $\zeta$ be the $\CLE_\kappa$ exploration interface of $\wh\Gamma$ from $x_3+r_3$ to $x_3-r_3$ up to the first time $\sigma$ it hits $I(x_4,r_4)$; see Figure~\ref{fig:exploration}. Let $\sigma'$ be the last time before $\sigma$ that $\zeta$ hits $I(x_3,r_3)$. Then conditioned on $\zeta[\sigma',\sigma]$, the restriction of $\wh\Gamma$ on $\cU(\zeta[\sigma',\sigma])$ is a $\CLE_\kappa$ on $\cU(\zeta[\sigma',\sigma])$ with two wired boundary arcs: $[x_1,x_2]$ and the outer boundary of $\zeta[\sigma',\sigma]$.

\begin{figure}[H]
\centering
\subfigure{
\includegraphics[width=0.45\textwidth]{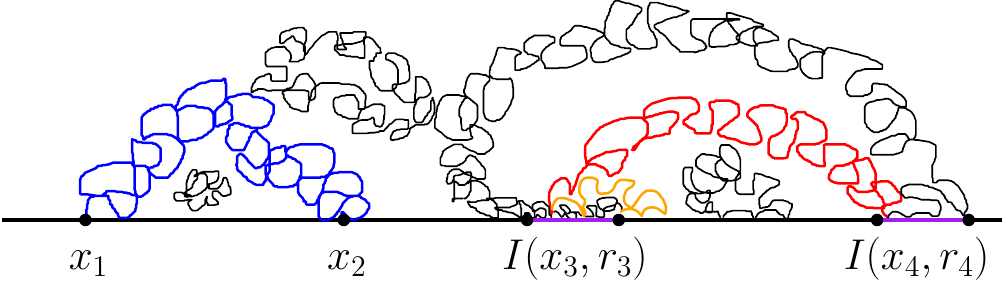}}
\hspace{15pt}
\subfigure{
\includegraphics[width=0.45\textwidth]{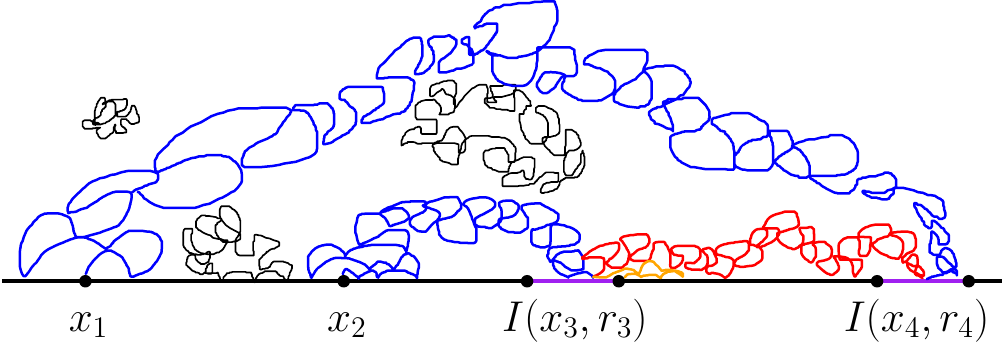}}
\caption{Illustration for the $\CLE_\kappa$ exploration interface $\zeta$ of $\wh\Gamma$. The segment $\zeta[\sigma,\sigma']$ is colored red, while $\zeta[0,\sigma']$ is in orange. \textbf{Left:} the event $E_{r_3,r_4}\setminus F_{r_3,r_4}$, and $\eta_{12}$ is colored blue. \textbf{Right:} the event $F_{r_3,r_4}$, and $\eta_{12}$ is the union of the blue and red curves.}
\label{fig:exploration}
\end{figure}

Let $F_{r_3,r_4}\subset E_{r_3,r_4}$ be the event that $\zeta[\sigma',\sigma]\subset\eta_{12}$. Note that given $E_{r_3,r_4}$, whether $F_{r_3,r_4}$ occurs or not gives a dichotomy of the two link patterns of $\wh\Gamma$. Let $\sm_{r_3,r_4}$ (resp.~$\sm'_{r_3,r_4}$) be the law of $\zeta[\sigma',\sigma]$ restricted on $F_{r_3,r_4}$ (resp.~$E_{r_3,r_4}\setminus F_{r_3,r_4}$). Then by~\cite[Theorem 1.1]{MW18}, $\sm_{r_3,r_4}$ and $\sm'_{r_3,r_4}$ are mutually absolutely continuous, with the Radon-Nikodym derivative
\begin{equation}\label{eq:mw-derivative}
\frac{d{\sm_{r_3,r_4}}}{d\sm_{r_3,r_4}'}(\eta)=
\frac{f(\tau)}{-2\cos\left(\frac{4\pi}{\kappa}\right) f(1-\tau)}.
\end{equation}
Here $\tau\in(0,1)$ is such that the $(\cU(\eta),x_1,x_2,x_3,x_4)$ is conformally equivalent to $(\hH,0,\tau,1,\infty)$.

On the other hand, note that $F_{r_3,r_4}$ implies $\eta_{12}$ intersecting both $I(x_3,r_3)$ and $I(x_4,r_4)$, while the intersection of the latter event and $E_{r_3,r_4}\setminus F_{r_3,r_4}$ yields a boundary three-arm event joining $I(x_3,r_3)$ and $I(x_4,r_4)$. Hence, by~\cite[Theorem 5.1]{zhan-boundary-gf}, we have $\lim_{r_3,r_4\to0}r_3^{-h}r_4^{-h}\P[F_{r_3,r_4}]=G_{\hH,x_1,x_2}(x_3,x_4)$. Furthermore, according to Lemma~\ref{lem:bubble-cle-variant} and conformal covariance, $r_3^{-h}r_4^{-h}\sm'_{r_3,r_4}$ weakly converges to $\int\mu_{\cU(\eta_{12}),x_3,x_4}(\cdot)\mu_{\hH,x_1,x_2}^\#(d\eta_{12})$\footnote{Here, using Lemma~\ref{lem:bubble-cle-variant} seems to involve some unspecified constant in the limiting measure. However, such constant is indeed fixed according to Remark~\ref{rmk:poisson}: by conformal covariance, we have $\lim_{r_4\to0}\lim_{r_3\to0}r_3^{-h}r_4^{-h}\P[E_{r_3,r_4}\setminus F_{r_3,r_4}]=\int H_{\cU(\eta_{12})}(x_3,x_4)^h\mu_{\hH,x_1,x_2}^\#(d\eta_{12})$.} (here the integration is taken over $\mu_{\hH,x_1,x_2}^\#(d\eta_{12})$) as $r_3\to0$ then $r_4\to0$.
Combined with~\eqref{eq:mw-derivative},
we find that $r_3^{-h}r_4^{-h}\sm_{r_3,r_4}$ weakly converges to the measure
\[
\int\frac{f(\tau)}{-2\cos\left(\frac{4\pi}{\kappa}\right) f(1-\tau)}\mu_{\cU(\eta_{12}),x_3,x_4}(\cdot)\mu_{\hH,x_1,x_2}^\#(d\eta_{12})=\frac{1}{-2\cos\left(\frac{4\pi}{\kappa}\right)} \frac{H_{\cU(\eta)}(x_1,x_2)^b}{H_{\hH}(x_1,x_2)^b}\tau ^{2b}f(\tau)\wt\mu_{\hH,x_3,x_4}(d\eta).
\]
Here to the right side we use Lemma~\ref{lem:cascade}. Combined with $r_3^{-h}r_4^{-h}|\sm_{r_3,r_4}|\to G_{\hH,x_1,x_2}(x_3,x_4)$, the result then follows.
\end{proof}

By conformal covariance, for any $\eta_{12}$ joining $x_1$ and $x_2$, Lemmas~\ref{lem:cascade} and~\ref{lem:cascade-2} imply
\begin{align}
\int H_{\cU(\eta')}(x_3,x_4)^h\mu^\#_{\cU(\eta_{12}),x_1,x_2}(d\eta')&=\int \tau^{2b}f(1-\tau)\Phi_{x_1,x_2}^b \wt\mu_{\cU(\eta_{12});x_3,x_4}(d\eta_{34}),\label{eq:12-34-cov}\\
G_{\cU(\eta_{12}),x_1,x_2}(x_4,x_3)
&=\frac{1}{-2\cos\left(\frac{4\pi}{\kappa}\right)}\int \tau^{2b}f(\tau)\Phi_{x_1,x_2}^b \wt\mu_{\cU(\eta_{12});x_3,x_4}(d\eta_{34}).\label{eq:1234-cov}
\end{align}
Here $\tau\in(0,1)$ is such that the unbounded connected component $R(\eta_{12},\eta_{34})$ of $\hH\setminus(\eta_{12}\cup\eta_{34})$, with four marked points $x_1,x_2,x_3,x_4$, is conformally equivalent to $(\hH,0,\tau,1,\infty)$; and $\Phi_{x_1,x_2}$ denotes the probability that the Brownian excursion on $\cU(\eta_{12})$ from $x_1$ to $x_2$ does not exit $R(\eta_{12},\eta_{34})$. Formally, we have $\Phi_{x_1,x_2}=\frac{H_{R(\eta_{12},\eta_{34})(x_1,x_2)}}{H_{\cU(\eta_{12})}(x_1,x_2)}=|\phi'(x_1)\phi'(x_2)|$, where $\phi:R(\eta_{12},\eta_{34})\to\cU(\eta_{12})$ is any conformal map fixing $x_1$ and $x_2$.

Combined with Lemma~\ref{lem:re-express}, we then have the following corollary for $G^{(1234)}(x_1,x_2,x_3,x_4)$ and $G^{(12)(34)}(x_1,x_2,x_3,x_4)$.
\begin{corollary}\label{cor:12-34}
For $x_1<x_2<x_3<x_4$, recall $f(x)$ defined in~\eqref{eq:f(x)}. Then
\begin{align}
&G^{(12)(34)}(x_1,x_2,x_3,x_4)=\int \tau^{2b}f(1-\tau)\Phi_{x_1,x_2}^b \wt\mu_{\cU(\eta_{12});x_3,x_4}(d\eta_{34})\wt \mu_{\hH,x_1,x_2}(d\eta_{12}),\label{eq:12-34-final}\\
&G^{(1234)}(x_1,x_2,x_3,x_4)=\frac{1}{-2\cos\left(\frac{4\pi}{\kappa}\right)}\int \tau^{2b}f(\tau)\Phi_{x_1,x_2}^b \wt\mu_{\cU(\eta_{12});x_3,x_4}(d\eta_{34})\wt \mu_{\hH,x_1,x_2}(d\eta_{12}).\label{eq:1234-final}
\end{align}
Here $\tau$ and $\Phi_{x_1,x_2}$ are defined as above.
\end{corollary}

Now we are ready to prove Proposition~\ref{prop:identify-v0}.
\begin{proof}[Proof of Proposition~\ref{prop:identify-v0}]
Recall that $x_1<x_2<x_3<x_4$. By Corollary~\ref{cor:12-34}, we have
\begin{equation}\label{eq:mix-expression}
G^{(1234)}(x_1,x_2,x_3,x_4)+G^{(12)(34)}(x_1,x_2,x_3,x_4)=\int \tau^{2b}Z(\tau) \Phi_{x_1,x_2}^b\wt\mu_{\cU(\eta_{12});x_3,x_4}(d\eta_{34})\wt \mu_{\hH,x_1,x_2}(d\eta_{12}),
\end{equation}
where $\tau$ and $\phi$ are defined in Corollary~\ref{cor:12-34}, and
\begin{equation}\label{eq:z-tau}
Z(\tau)=f(1-\tau)+\frac{1}{-2\cos\left(\frac{4\pi}{\kappa}\right)}f(\tau)=\tau^{-2b}(1+O(\tau^2)).
\end{equation}
We emphasize that the error term $O(\tau^2)$ in~\eqref{eq:z-tau} is crucial to our proof. As we noted in Section~\ref{sec:discussion}, the conjectural partition functions of general multichordal $\CLE_\kappa$ also exhibit the same rapid decay of subleading terms~\cite[Lemma 6.1]{feng2024multiplesle}.

In the following, we use~\eqref{eq:mix-expression} and~\eqref{eq:z-tau} to show that as $x_2\to x_1$ (with $x_1,x_3,x_4$ fixed), the summation of Green's function $G^{(1234)}(x_1,x_2,x_3,x_4)+G^{(12)(34)}(x_1,x_2,x_3,x_4)$ is asymptotically equal to $ H_{\hH}(x_1,x_2)^h H_{\hH}(x_3,x_4)^h(1+o(|x_2-x_1|^h))$. To this end, we choose $c_1,c_2$ such that\[
(\frac{24}{\kappa}-2)(1-c_1)>h,\ (\frac{24}{\kappa}-2)c_2>h,\ 2(c_1-c_2)>h\ \text{with}\ c_1>\frac{1}{2}>c_2.
\]
This is always possible when $\kappa\in(4,8)$. Let $|x_2-x_1|=\varepsilon$, and $E$ be the event that ${\rm diam}(\eta_{12})\le\varepsilon^{c_1}$ while ${\rm dist}(x_1,\eta_{34})\ge\varepsilon^{c_2}$. Note that on the event $E$, by basic conformal distortion estimates, we have $\tau=O(\varepsilon^{c_1-c_2})$ and $\Phi_{x_1,x_2}=1+O(\varepsilon^{2(c_1-c_2)})$. Meanwhile, the conformal restriction property of $\SLE_\kappa(2)$ (see~\cite{dubedat-rho}) gives that for $(\eta_{12},\eta_{34})\in E$, we have the Radon-Nikodym derivative $\frac{d\wt\mu_{\cU(\eta_{12});x_3,x_4}}{d\wt\mu_{\hH;x_3,x_4}}[\eta_{34}]=1+O(\varepsilon^{2(c_1-c_2)})$. The boundary one-point estimate for $\SLE_\kappa(2)$ (see~\cite[Theorem 4.1]{zhan-boundary-gf}) gives
\[
\wt\mu_{\hH;x_1,x_2}^\#[{\rm diam}(\eta_{12})>\varepsilon^{c_1}]=O(\varepsilon^{(\frac{24}{\kappa}-2)(1-c_1)}),\quad \wt\mu_{\hH;x_3,x_4}^\#[{\rm dist}(x_1,\eta_{34})<\varepsilon^{c_2}]=O(\varepsilon^{(\frac{24}{\kappa}-2)c_2}),
\]
where $\frac{24}{\kappa}-2=3h+1$ corresponds to the boundary three-arm exponent for $\SLE_\kappa$. Thus,
\begin{align*}
\wt\mu_{\hH;x_1,x_2}[{\rm diam}(\eta_{12})\le\varepsilon^{c_1}]&=H_{\hH}(x_1,x_2)^h(1+O(\varepsilon^{(\frac{24}{\kappa}-2)(1-c_1)})),\\
\wt\mu_{\hH;x_3,x_4}[{\rm dist}(x_1,\eta_{34})\ge\varepsilon^{c_2}]&=H_{\hH}(x_3,x_4)^h(1+O(\varepsilon^{(\frac{24}{\kappa}-2)c_2})).
\end{align*}
Combined with~\eqref{eq:z-tau}, we have
\begin{equation}\label{eq:main-part}
\int_E \tau^{2b}Z(\tau)\Phi_{x_1,x_2}^b \wt\mu_{\cU(\eta_{12});x_3,x_4}(d\eta_{34})\wt \mu_{\hH,x_1,x_2}(d\eta_{12})= H_{\hH}(x_1,x_2)^hH_{\hH}(x_3,x_4)^h(1+o(\varepsilon^h)).
\end{equation}

We now consider the complement $E^c$ of the event $E$. First, note that by~\eqref{eq:12-34-cov},
\begin{align*}
&\int_{{\rm diam}(\eta_{12})>\varepsilon^{c_1}} \tau^{2b}f(1-\tau)\Phi_{x_1,x_2}^b \wt\mu_{\cU(\eta_{12});x_3,x_4}(d\eta_{34})\wt \mu_{\hH,x_1,x_2}(d\eta_{12})\\
&=\int_{{\rm diam}(\eta_{12})>\varepsilon^{c_1}} H_{\cU(\eta')}(x_3,x_4)^h\mu^\#_{\cU(\eta_{12}),x_1,x_2}(d\eta')\wt\mu_{\hH,x_1,x_2}(d\eta_{12})\\
&\le H_{\hH}(x_3,x_4)^hH_{\hH}(x_1,x_2)^h\wt\mu_{\hH,x_1,x_2}^\#[{\rm diam}(\eta_{12})>\varepsilon^{c_1}]= H_{\hH}(x_1,x_2)^hH_{\hH}(x_3,x_4)^ho(\varepsilon^h)
\end{align*}
due to the monotonicity of the boundary Poisson kernel and the same boundary 1-point estimate for $\SLE_\kappa(2)$ from~\cite{zhan-boundary-gf} as above. The event ${\rm dist}(x_1,\eta_{34})>\varepsilon^{c_2}$ is similar. Thus, we have
\begin{equation}\label{eq:error-12,34}
\int_{E^c} \tau^{2b}f(1-\tau)\Phi_{x_1,x_2}^b \wt\mu_{\cU(\eta_{12});x_3,x_4}(d\eta_{34})\wt \mu_{\hH,x_1,x_2}(d\eta_{12})= H_{\hH}(x_1,x_2)^hH_{\hH}(x_3,x_4)^ho(\varepsilon^h).
\end{equation}

It remains to deal with $\int_{E^c} \tau^{2b}f(\tau)\Phi_{x_1,x_2}^b \wt\mu_{\cU(\eta_{12});x_3,x_4}(d\eta_{34})\wt \mu_{\hH,x_1,x_2}(d\eta_{12})$. By~\eqref{eq:1234-cov}, we have
\begin{align*}
&\frac{1}{-2\cos\left(\frac{4\pi}{\kappa}\right)}\int_{{\rm diam}(\eta_{12})>\varepsilon^{c_1}} \tau^{2b}f(\tau)\Phi_{x_1,x_2}^b \wt\mu_{\cU(\eta_{12});x_3,x_4}(d\eta_{34})\wt \mu_{\hH,x_1,x_2}(d\eta_{12})\\
&=\int_{{\rm diam}(\eta_{12})>\varepsilon^{c_1}}G_{\cU(\eta_{12}),x_1,x_2}(x_4,x_3)\wt \mu_{\hH,x_1,x_2}(d\eta_{12}),
\end{align*}
where $G_{D,a,b}(\cdot,\cdot)$ stands for the boundary two-point Green's function for chordal $\SLE_\kappa$ on $D$ from $a$ to $b$. Fix $\delta_0=\frac{1}{10}\min(|x_3-x_1|,|x_4-x_3|)$.
Let $0<\delta<\delta_0$ and $0\le k_1,k_2\le K:=\lfloor\log_2\frac{\delta_0}{\delta}\rfloor$. Denote
\[
F_{k_1,k_2}:=\{{\rm dist}(x_3,\eta_{12})\in[2^{k_1}\delta,2^{k_1+1}\delta]\text{ and } {\rm dist}(x_4,\eta_{12})\in[2^{k_2}\delta,2^{k_2+1}\delta]\}.
\]
We claim that on the event $F_{k_1,k_2}$,
\begin{equation}\label{eq:gf-estimate}
G_{\cU(\eta_{12}),x_1,x_2}(x_4,x_3)\le C (2^{k_1}\delta)^{-h}(2^{k_2}\delta)^{-h}.
\end{equation}
Here and after, $C>0$ is some constant depending on $x_1,x_3,x_4$ and can vary from line to line. To see~\eqref{eq:gf-estimate}, first note that for any compact hull $A$, $A\cap\R=[a,b]$, $x>b$ such that ${\rm dist}(x,A)>d_1>0$, for the conformal map $\psi:\hH\setminus A\to\hH$ with $\psi(a)=0$, $\psi(b)=\infty$ and $\psi(x)=1$, we have $|\psi'(x)|\le 4d_1^{-1}$ by Koebe's 1/4 theorem. By conformal covariance, this implies the boundary one-point Green's function $G_{\hH\setminus A,a,b}(x)\le C d_1^{-h}$. Now suppose $b<y<x$, ${\rm dist}(y,A)>d_2>0$ and $|x-y|>d_1+d_2$. According to the martingale property of the boundary two-point Green's function (see~\cite[Theorem 4.1]{fakhry2023}), we have
\[
G_{\hH\setminus A,a,b}(x,y)=G_{\hH\setminus A,a,b}(x)\E_x[G_{U^*,x,b}(y)]
\]
here $\E_x$ is with respect to the conditional probability measure $\P_x$ of the chordal $\SLE_\kappa$ curve $\eta$ on $\hH\setminus A$ from $a$ to $b$ conditioned to hit $x$ (we denote this hitting time by $\sigma$), and $U^*$ is the connected component of $\hH\setminus(A\cup\eta[0,\sigma])$ such that $b$ is on its boundary. Furthermore, following the estimate in~\cite[Lemma 3.5]{fakhry2023}, for $s\in(0,1)$, the probability that a chordal $\SLE_\kappa$ on $\hH$ from $0$ to $\infty$ hits $\partial B(1,\frac{1}{10})$ after hitting $B(1,\frac{1}{10}s)$ and before disconnecting $1$ from $\infty$ is bounded above by $Cs^{2h}$. The domain Markov property of $\SLE_\kappa$ (see e.g.~\cite[Eq.(3.27)]{fakhry2023}) then implies $\P_x[{\rm dist}(U^*,y)<r d]\le Cr^{2h}$ for $r\in(0,\frac{1}{10})$. Consequently, we obtain $G_{\hH\setminus A,a,b}(x,y)\le Cd_1^{-h}d_2^{-h}$, thus proving~\eqref{eq:gf-estimate}.

Applying the boundary two-point estimate~\cite[Lemma 5.2]{zhan-boundary-gf} for $\SLE_\kappa(2)$ gives 
\[
\wt \mu_{\hH,x_1,x_2}^\#[F_{k_1,k_2}]\le C(\varepsilon 2^{k_1}\delta)^{\frac{24}{\kappa}-2}(2^{k_2}\delta)^{\frac{24}{\kappa}-2},\quad 1\le k_1,k_2\le K, 
\]
and $\wt \mu_{\hH,x_1,x_2}^\#[{\rm diam}(\eta_{12})>\varepsilon^{c_1}]\le C\varepsilon^{(\frac{24}{\kappa}-2)(1-c_1)}$ as above. Also note that on the event that ${\rm dist}(x_3,\eta_{12})>\delta_0$ and ${\rm dist}(x_4,\eta_{12})>\delta_0$, $G_{\cU(\eta_{12}),x_1,x_2}(x_4,x_3)$ is bounded above by some constant $C$. Combining these with~\eqref{eq:gf-estimate}, we find
\begin{align*}
&\int_{{\rm diam}(\eta_{12})>\varepsilon^{c_1},{\rm dist}(x_3,\eta_{12})>\delta,{\rm dist}(x_4,\eta_{12})>\delta}G_{\cU(\eta_{12}),x_1,x_2}(x_4,x_3)\wt \mu_{\hH,x_1,x_2}(d\eta_{12})\\
&\le C H_{\hH}(x_1,x_2)^h\left(\sum_{k_1,k_2=0}^{K}\mu_{\hH,x_1,x_2}^\#[F_{k_1,k_2}] (2^{k_1}\delta)^{-h}(2^{k_2}\delta)^{-h}+\wt \mu_{\hH,x_1,x_2}^\#[{\rm diam}(\eta_{12})>\varepsilon^{c_1}]\right)\\
&\le C H_{\hH}(x_1,x_2)^h\left(\sum_{k_1,k_2=0}^{K}\varepsilon^{\frac{24}{\kappa}-2}(2^{k_1}\delta)^{\frac{16}{\kappa}-1}(2^{k_2}\delta)^{\frac{16}{\kappa}-1}+\varepsilon^{(\frac{24}{\kappa}-2)(1-c_1)}\right)\\
&\le CH_{\hH}(x_1,x_2)^h\varepsilon^{(\frac{24}{\kappa}-2)(1-c_1)}= H_{\hH}(x_1,x_2)^ho(\varepsilon^h).
\end{align*}
In particular, the error term $o(\varepsilon^h)$ does not depend on $\delta$. Since $\delta$ is arbitrary, this gives
\[
\int_{{\rm diam}(\eta_{12})>\varepsilon^{c_1}}G_{\cU(\eta_{12}),x_1,x_2}(x_4,x_3)\wt \mu_{\hH,x_1,x_2}(d\eta_{12})= H_{\hH}(x_1,x_2)^ho(\varepsilon^h).\]
The case ${\rm dist}(x_1,\eta_{34})>\varepsilon^{c_2}$ follows similarly by considering $\int_{{\rm dist}(x_1,\eta_{34})>\varepsilon^{c_2}}G_{\cU(\eta_{34}),x_4,x_3}(x_1,x_2)\wt \mu_{\hH,x_3,x_4}(d\eta_{34})$, and dividing dyadic scales at $x_1,x_2$. Therefore, we obtain
\begin{equation}\label{eq:error-1234}
\frac{1}{-2\cos\left(\frac{4\pi}{\kappa}\right)}\int_{E^c} \tau^{2b}f(\tau)\Phi_{x_1,x_2}^b \wt\mu_{\cU(\eta_{12});x_3,x_4}(d\eta_{34})\wt \mu_{\hH,x_1,x_2}(d\eta_{12})= H_{\hH}(x_1,x_2)^ho(\varepsilon^h).
\end{equation}

Combining~\eqref{eq:main-part},~\eqref{eq:error-12,34} and~\eqref{eq:error-1234}, we find
\begin{equation}\label{eq:estimation-final}
\int \tau^{2b}Z(\tau)\Phi_{x_1,x_2}^b \wt\mu_{\cU(\eta_{12});x_3,x_4}(d\eta_{34})\wt \mu_{\hH,x_1,x_2}(d\eta_{12})= H_{\hH}(x_1,x_2)^hH_{\hH}(x_3,x_4)^h(1+o(\varepsilon^h))
\end{equation}
as $\varepsilon=|x_2-x_1|\to0$ (with $x_1,x_3,x_4$ fixed). The result then follows from~\eqref{eq:mix-expression} and~\eqref{eq:estimation-final}.
\end{proof}

\subsection{Proof of Theorems~\ref{thm:percolation} and~\ref{thm:fk}}\label{sec:kappa=6}\label{sec:kappa=16/3}

Theorems~\ref{thm:percolation} and~\ref{thm:fk} now follow as consequences of Theorems~\ref{thm:ode} and~\ref{thm:identify} when $\kappa=6$ and $\kappa=\frac{16}{3}$.
\begin{proof}[Proof of Theorem~\ref{thm:percolation}]
Taking $\kappa=6$ into~\eqref{eq:ode}, we obtain
\begin{equation}\label{eq:ode-6}
9\lambda^2(1-\lambda)^2U'''+ 6\lambda(1-\lambda)(1-2\lambda)U''+ (8\lambda(1-\lambda)-6)U'+4(2\lambda-1)U =0.
\end{equation}
Then~\eqref{eq:ode-6} has three linearly independent solutions $V_0(\lambda)$, $V_2(\lambda)$ and $V_{\frac{1}{3}}(\lambda)$, which are defined in Section~\ref{sec:intro-perc}. See also~\cite[Appendix C]{Gori:2018gqx}. In particular, $V_0(\lambda)=V_0(1-\lambda)$, and as $\lambda\to0$,
\[
V_0(\lambda)=V_0(1-\lambda)=1-\frac{2}{3}\lambda+O(\lambda^2|\log\lambda|),\quad V_2(\lambda)=O(\lambda^2),\quad V_{2}(1-\lambda)=C_0+O(\lambda^{\frac{1}{3}})
\]
for some constant $C_0\in(0,\infty)$. Therefore, by Theorem~\ref{thm:identify}, we have $U^{\rm total}(\lambda)=V_0(\lambda)$, $U^{(14)(23)}(\lambda)=V_2(\lambda)$ and $U^{(12)(34)}(\lambda)=V_2(1-\lambda)$. Namely, for some $C,C'\in(0,\infty)$,
\begin{equation*}
\begin{aligned}
G^{\rm total}(x_1,x_2,x_3,x_4)&=C\left(\frac{(x_4-x_2)(x_3-x_1)}{(x_2-x_1)(x_4-x_3)(x_3-x_2)(x_4-x_1)}\right)^{\frac{2}{3}}V_0(\lambda),\\
G^{(14)(23)}(x_1,x_2,x_3,x_4)&=C'\left(\frac{(x_4-x_2)(x_3-x_1)}{(x_2-x_1)(x_4-x_3)(x_3-x_2)(x_4-x_1)}\right)^{\frac{2}{3}}V_2(\lambda).
\end{aligned}
\end{equation*}
On the other hand, standard discrete argument (see Proposition~\ref{prop:discrete}) shows that $P^p(x_1,x_2,x_3,x_4)$ agrees with $G^p(x_1,x_2,x_3,x_4)$ for $p\in\{(1234),(12)(34),(14)(23),{\rm total}\}$ up to a multiplicative constant. The coefficient $A=\frac{8\sqrt{3}\,\pi \sin\left(\frac{2\pi}{9}\right)}{135\cos\left(\frac{5\pi}{18}\right)}$ is determined by $\lim_{x_2\to x_3}\frac{P^{(14)(23)}(x_1,x_2,x_3,x_4)}{P^{\rm total}(x_1,x_2,x_3,x_4)}=1$. We conclude.
\end{proof}

\begin{remark}
For $\kappa=6$, the identification of $U^{\rm total}$ with $V_0(\lambda)$ can also be obtained via analyzing the asymptotic behavior of $P^{\rm total}(x_1,x_2,x_3,x_4)$ as $|x_1-x_2|\to0$ from the discrete side; see~\cite[Theorem 1.9]{cf2025}. Indeed,~\cite[Theorem 1.9]{cf2025} gives that
\[
P^{\rm total}(x_1,x_2,x_3,x_4)\propto|x_1-x_2|^{-\frac{2}{3}}(1+O(|x_1-x_2|^2|\log|x_1-x_2||)),
\]
as $x_2\to x_1$ (with $x_1,x_3,x_4$ fixed). This provides an alternative proof of Proposition~\ref{prop:identify-v0} for $\kappa=6$.
\end{remark}

\begin{proof}[Proof of Theorem~\ref{thm:fk}]
Taking $\kappa=\frac{16}{3}$ into~\eqref{eq:ode}, we obtain
\begin{equation}\label{eq:ode-16/3}
4\lambda^2(1-\lambda)^2U'''-3(\lambda^2-\lambda+1)U'+3(2\lambda-1)U=0. 
\end{equation}
For~\eqref{eq:ode-16/3}, we have a special solution $V_0(\lambda)=1-\lambda+\lambda^2$. Let $U(\lambda)=V_0(\lambda)\int_0^{\lambda}g(x)dx$, then~\eqref{eq:ode-16/3} yields a second-order ODE for $g$, whose solutions involve hypergeometric functions. As a result,
$V_{5/2}(\lambda)=V_0(\lambda)\int_0^{\lambda}g(x)dx$, where
\begin{equation*}
g(x) = -\frac{(1 - x)^{3/2} \left( (2 - 4x) \, {}_2F_1\!\left(\frac{3}{2}, \frac{7}{2}; 3; 1 - x\right) - 3x(1 - x) \, {}_2F_1\!\left(\frac{5}{2}, \frac{9}{2}; 4; 1 - x\right) \right) x^{3/2}}{2\left(x + (1 - x)^2\right)^2},
\end{equation*}
as in~\eqref{eq:fk-g}.
Note that $g(x)=x^\frac{3}{2}(1+O(x))$ as $x\to0$ (hence $V_{5/2}(\lambda)=\lambda^{\frac{5}{2}}(1+O(\lambda))$ as $\lambda\to0$). Since $V_0(\lambda)=V_0(1-\lambda)$, by Theorem~\ref{thm:identify}, we have for some $C,C'\in(0,\infty)$,
\begin{equation*}
\begin{aligned}
G^{(14)(23)}(x_1,x_2,x_3,x_4)&=C\left(\frac{(x_4-x_2)(x_3-x_1)}{(x_2-x_1)(x_4-x_3)(x_3-x_2)(x_4-x_1)}\right)V_{5/2}(\lambda).\\
G^{\rm total}(x_1,x_2,x_3,x_4)&=C'\left(\frac{(x_4-x_2)(x_3-x_1)}{(x_2-x_1)(x_4-x_3)(x_3-x_2)(x_4-x_1)}\right) V_0(\lambda).
\end{aligned}
\end{equation*}
Similar to the Bernoulli percolation case (see the end of Appendix~\ref{appendix:discrete}), $P^p_{\rm FK}(x_1,x_2,x_3,x_4)$ agrees with $G^p(x_1,x_2,x_3,x_4)$ for $p\in\{(1234),(12)(34),(14)(23),{\rm total}\}$ up to a multiplicative constant. Therefore, the universal ratio $R_{\rm FK}(\lambda)$ is
\begin{equation*}
R_{\rm FK}(\lambda)=A_{\rm FK}\frac{V_{5/2}(\lambda)}{V_{0}(\lambda)}=A_{\rm FK}\int_0^{\lambda}g(x)dx
\end{equation*}
where $A_{\rm FK}=(\int_0^{1}g(x)dx)^{-1}\simeq1.19948$ is such that $R_{\rm FK}(1)=1$, as desired.
\end{proof}
\begin{remark}
According to the Edwards-Sokal coupling, note that $P_{\rm FK}^{\rm total}(x_1,x_2,x_3,x_4)$ equals the boundary four-point spin correlation of the critical Ising model with free boundary condition, which can be formally obtained by taking limits of the bulk spin correlations derived in~\cite{CHI15}. This also explains the reason that $G^{\rm total}(x_1,x_2,x_3,x_4)$ has a rather simple form when $\kappa=\frac{16}{3}$.
\end{remark}

\subsection{Solutions of~\eqref{eq:ode} for other special $\kappa$}\label{sec:kappa-le-4}

In this section we discuss solutions of~\eqref{eq:ode} for other special $\kappa$'s.

\begin{itemize}
    \item $\kappa=\frac{24}{5}$. This corresponds to the conjectural scaling limit of 3-Potts model. In this case,~\eqref{eq:ode} has rather simple solutions: we have $V_0(\lambda)=1-\frac{4}{3}\lambda+\frac{4}{3}\lambda^2=V_0(1-\lambda)$, $V_{2/3}(\lambda)=\lambda^{2/3}(1-\lambda+\frac{3}{4}\lambda^2)$, and $V_{3}(\lambda)=V_0(\lambda)-\frac{4}{3}V_{2/3}(1-\lambda)$. See~\cite[Section 4]{Gori:2018gqx}. Thus, according to Theorem~\ref{thm:identify}, we have $U^{(14)(23)}\propto V_3$, and $U^{\rm total}\propto V_0$.
\item $\kappa=8$. This corresponds to the scaling limit of the uniform spanning tree (UST, with free boundary condition). The general solution of~\eqref{eq:ode} now is $ U(\lambda) = c_1 + c_2 |\ln(1-\lambda)| + c_3 |\ln\lambda|$. The constant solution is clearly due to that every four boundary points are connected in the UST. We conjecture that the remaining solutions could be viewed e.g.~as the partition function of two trees such that their union spans all vertices, with $x_1,x_2$ connected in one tree and $x_3,x_4$ connected in the other tree.
\end{itemize}

We expect that~\eqref{eq:ode} also makes sense for $\kappa\in(0,4]$. Let $\SLE_{\kappa,\hH}^\lp$ be the $\SLE_\kappa$ loop measure~\cite{zhan-loop-measures} on $\hH$. Then the solutions of~\eqref{eq:ode} would be the following counterparts of~\eqref{eq:bub-2} and~\eqref{eq:bub-3}: $G^{(1234)}$ corresponds to the boundary four point Green's function of the $\SLE_\kappa$ loop measure, defined via the $\varepsilon\to0$ limit of $\varepsilon^{-4h}\SLE_{\kappa,\hH}^\lp[\ell\cap B(x_i,\varepsilon)\neq\emptyset]$ (assuming the limit exists); $G^{(12)(34)}$ corresponds to the integral $\int H_{\cU(\gamma)}(x_3,x_4)^h\mu_{x_1,x_2}^\bub(d\gamma)$ , where $\mu_{x_1,x_2}^\bub$ is the $\SLE_\kappa$ bubble measure rooted at $x_1$ and $x_2$~\eqref{eq:bubble-two-pinned}, and $h=\frac{8}{\kappa}-1\ge1$. By conformal covariance, we can also introduce the function $U$'s as in~\eqref{eq:u-lambda}. Of course, it needs extra effort to establish the existence of these quantities, and show that they satisfy~\eqref{eq:ode}. Furthermore, it requires the $\kappa\in(0,4]$ counterpart of Proposition~\ref{prop:identify-v0} in order to recover the Green's functions from linear combinations of these solutions. Here we include solutions of~\eqref{eq:ode} for several $\kappa\in(0,4]$.

\begin{itemize}
\item $\kappa=4$. This was also included in~\cite[Section 4]{Gori:2018gqx}, as the conjectural scaling limit of 4-Potts model. The general solution is $U(\lambda)=c_1+c_2(\lambda-\frac{3}{2}\lambda^2+\lambda^3)+c_3\lambda^4$.
\item $\kappa=\frac{8}{3}$. This corresponds to partition functions of Brownian excursions.~\eqref{eq:ode} then reads as
\[
\lambda^2 (1-\lambda)^2 U''' -6 \lambda (1-\lambda)(1-2\lambda) U'' + [ 6 + 48 \lambda(\lambda-1) ] U' -24 (2\lambda-1) U =0.
\]
Note that $U_0=\lambda^2(1-\lambda)^2(1+\lambda^2+(1-\lambda)^2)$ solves the above ODE, which corresponds to the summation of the total mass of four Brownian excursions joining $0,\lambda,1,\infty$ in arbitrary orders such that the union forms a loop (i.e.~the $\kappa=\frac{8}{3}$ counterpart of $G^{(1234)}$). Taking $U=U_0\int_0^\lambda w(x)dx$ into the equation, we can then find the other two solutions
\[
U_1=\Bigg( 5\lambda^2 - 5\lambda - \frac{5}{(\lambda - 1)^2} - \frac{5}{\lambda - 1} - 24\ln(1-\lambda ) + \frac{7(-\lambda - 1)}{\lambda^2 - \lambda + 1} +7\Bigg)U_0,\quad U_2(\lambda)=U_1(1-\lambda).
\]
They correspond to $G^{(14)(23)}$ and $G^{(12)(34)}$, respectively, and have the following probabilistic interpretation. Let $e_1$ (resp.~$e_2$) be the union of \emph{two} independent Brownian excursions from $0$ to $\infty$ (resp.~from $\lambda$ to $1$) such that $e_1,e_2$ are also independent. Let $P(\lambda)$ be the probability that $e_1$ does not intersect with $e_2$. Then we have $P(\lambda)\propto \frac{U_1(\lambda)}{\lambda^4}$, i.e.
\[
P(\lambda)=-\frac{(1-\lambda)^2}{5\lambda^2}(\lambda^2-\lambda+1)\Bigg( 5\lambda^2 - 5\lambda - \frac{5}{(\lambda - 1)^2} - \frac{5}{\lambda - 1} - 24\ln(1-\lambda ) + \frac{7(-\lambda - 1)}{\lambda^2 - \lambda + 1} +7\Bigg).
\]
This expression can be proved similarly to Section~\ref{sec:fusion}, using the fact that the right boundary of $e_1$, which has the law of $\SLE_{8/3}(2)$ (see~\cite{lsw-restriction}), is the fusion limit of a pair of $\SLE_{8/3}$ from $0$ to $\infty$, as well as the half-plane intersection exponent $P(\lambda)=\lambda^{3+o(1)}$ as $\lambda\to0$~\cite{lsw-bm-exponents1}.
\item $\kappa=2$. This is related to the loop-erased random walk (LERW).
In this case, the three linearly independent solutions are $U_0(\lambda)=6\lambda^2 - 6\lambda + 1$, $U_1(\lambda)=\lambda^{10}(\lambda^2-6\lambda+6)$ and $U_2(\lambda)=U_1(1-\lambda)$.
We conjecture that $U_2(\lambda)$ gives the $\kappa=2$ counterpart of $G^{(12)(34)}$, while further efforts would be needed to identify $G^{(1234)}$ with a linear combination of the above solutions.
\end{itemize}

\section{One-Bulk and two-boundary connectivities}\label{sec:bulk-boundary}

Our framework also works for the one-bulk and two-boundary connectivities, which proves Theorem~\ref{thm:bulk}. Recall that for $x_1,x_2\in\R$ and $z\in\hH$, its one-bulk and two-boundary Green's function is defined in~\eqref{eq:def-gf-bulk}.
By Proposition~\ref{prop:touching-equals-bubble} and the conformal covariance of the Miller-Schoug measure,~\eqref{eq:def-gf-bulk} is equivalent to
\begin{equation}\label{eq:CR}
G(x_1,x_2,z)=C\int {\bf 1}_{z \text{ is surrounded by } \gamma} \CR(z,D_z(\gamma))^{-\alpha} \mu_{x_1,x_2}^\bub(d\gamma),
\end{equation}
for some $C\in(0,\infty)$. Here, $\mu_{x_1,x_2}^\bub(d\gamma)$ is the $\SLE_\kappa$ bubble measure rooted at $x_1$ and $x_2$~\eqref{eq:bubble-two-pinned}, $\CR(z,D_z(\gamma))$ is the conformal radius of the connected component $D_z(\gamma)$ of $\hH\setminus\gamma$ containing $z$, viewed at $z$, and ``surround'' refers to that the winding number of $\gamma$ around $z$ is non-zero.

\begin{proof}[Proof of Theorem~\ref{thm:bulk}]
For $x_1,x_2\in\R$ and $z\in\hH$, define
\[
\lambda=\frac{(x_2-x_1)(\bar z-z)}{(z-x_1)(\bar z-x_2)}.
\]
Note that $\lambda\in\C$ and $\bar\lambda=\frac{\lambda}{\lambda-1}$, hence it takes value on the circle $\{\lambda\in\C:|\lambda-1|=1\}$.
By conformal covariance, there is a real-valued function $\Delta(\lambda)$ such that
\[
G(x_1,x_2,z)=|x_2-x_1|^{-2h}|z-\bar z|^{-\alpha}\Delta(\lambda).
\]
Similar to Proposition~\ref{prop:gf-convergence-sec2}, we can use the chordal $\SLE_\kappa$ to approximate the Green's function~\eqref{eq:def-gf-bulk}. Namely, define
\[
F(y_1,y_2,x,z)dx=\int{\bf 1}_{E_z(\eta)}\CR(z,D_z(\eta))^{-\alpha}\nu_{\eta\cap\R}(dx)\mu_{\hH,y_1,y_2}(d\eta),
\]
where $D_z(\gamma)$ is the connected component  of $\hH\setminus\gamma$ containing $z$, $\CR(z,D_z(\eta))$ is its conformal radius viewed at $z$, and $E_z(\eta)$ is the event that $z$ is surrounded by the joining of $\gamma$ and the line segment from $y_2$ to $y_1$.
Then we have $G(u,x,z)\propto\lim_{y_1,y_2\to u} |y_2-y_1|^{-h}F(y_1,y_2,x,z)$.
Furthermore, when $\eta$ is parameterized by its half-plane capacity, for its Loewner map $g_t:\hH\setminus\eta_t\to\hH$, we have
\[
M_t:=|g_t'(x)|^h|g_t'(z)|^\alpha F(g_t(y_1),g_t(y_2),g_t(x),g_t(z))
\]
is a local martingale. The same argument as in Lemma~\ref{lem:smooth} gives the smoothness of $F$. Hence, $F$ satisfies the following second-order PDE
\begin{equation}\label{eq:pde-bulk}
\begin{aligned}
\left(-\frac{2h}{(x-y_1)^2}\right.-\frac{\alpha}{(z-y_1)^2}&-\frac{\alpha}{(\bar z-y_1)^2}-\frac{\kappa-6}{y_2-y_1}\partial_{y_1}+\frac{2}{y_2-y_1}\partial_{y_2}\\
&\left.+\frac{2}{x-y_1}\partial_{x_1}+\frac{2}{z-y_1}\partial_{z}+\frac{2}{\bar z-y_1}\partial_{\bar z}+\frac{\kappa}{2}\partial_{y_1}^2\right)F=0.
\end{aligned}  
\end{equation}
By symmetry, the above PDE also holds when changing the place of $y_1$ and $y_2$. Here for convenience, we view $z$ and $\bar z$ as independent variables. Note that such equations are similar to those in Proposition~\ref{prop:pde-u-v}, with the three marked points $(x_1,h), (x_2,h), (x_3,h)$ replaced by $(x,h), (z,\frac{\alpha}{2}), (\bar z,\frac{\alpha}{2})$. Consequently, we can repeat the fusion procedure in Section~\ref{sec:fusion} in parallel. Namely, \eqref{eq:pde-bulk} implies a third-order differential equation for $G(u,x,z)$. The resulting ODE\footnote{Since $\lambda\in\C$ takes value on the circle $\{\lambda\in\C:|\lambda-1|=1\}$, we can view $\lambda$ as if it were a real variable.} for $\Delta(\lambda)$ is
\begin{equation}\label{eq:fusion-ode-bulk}
\begin{aligned}
&\kappa^2\lambda^2(1 - \lambda)^3\Delta'''(\lambda)\\
&-2\kappa\lambda(1 - \lambda)^2((3\kappa-8)\lambda-(3\kappa-16))\Delta''(\lambda)\\
&+(1 - \lambda)((-8\alpha\kappa + 6\kappa^2 - 40\kappa + 64)\lambda^2  -4(\kappa - 6)(3\kappa - 8)\lambda + 6(\kappa - 4)(\kappa-8))\Delta'(\lambda)\\
&+8\alpha(8-\kappa)\lambda(2-\lambda)\Delta(\lambda)= 0.
\end{aligned}
\end{equation}

Note that $\lambda_0=2$ is an ordinary point of~\eqref{eq:fusion-ode-bulk}. Taking $\alpha=\frac{(3\kappa-8)(8-\kappa)}{32\kappa}$ into~\eqref{eq:fusion-ode-bulk} and solving~\eqref{eq:fusion-ode-bulk} at $\lambda_0$, the general solution in a neighborhood of $\lambda_0$ is
\begin{equation}\label{eq:solution-bulk}
\begin{aligned}
\Delta(\lambda) = c_1 \, \lambda^{-\frac{\kappa-8}{\kappa}} (1-\lambda)^{\frac{\kappa-8}{2\kappa}}&+ c_2 \, (1-\lambda)^{\frac{\kappa-8}{\kappa}} \, {}_2F_1\!\left(\frac{2(\kappa-8)}{\kappa}, \frac{3(\kappa-8)}{2\kappa}; \frac{3\kappa-8}{2\kappa}; 1-\lambda\right) \\
&+ c_3 \, (1-\lambda)^{\frac{\kappa-8}{2\kappa}} \, {}_2F_1\!\left(\frac{\kappa-8}{\kappa}, \frac{3(\kappa-8)}{2\kappa}; \frac{\kappa+8}{2\kappa}; 1-\lambda\right)
\end{aligned}
\end{equation}
where $c_1,c_2,c_3\in\C$. Indeed,~\eqref{eq:solution-bulk} is well-defined and smooth on $\{\lambda\in\C:|\lambda-1|=1,\lambda\neq0\}$, thus provides the general solution on $\{\lambda\in\C:|\lambda-1|=1,\lambda\neq0\}$. Furthermore, $\Delta(\lambda)$ needs to be $o(1)$
both as $|x_1-x_2|\to0$ (with $z$ fixed) and as ${\rm Im}(z)\to0$ (with $x_1,x_2$ and ${\rm Re}(z)$ fixed).
Since both $(1-\lambda)^{\frac{\kappa-8}{\kappa}} \, {}_2F_1\!\left(\frac{2(\kappa-8)}{\kappa}, \frac{3(\kappa-8)}{2\kappa}; \frac{3\kappa-8}{2\kappa}; 1-\lambda\right)
$ and $ (1-\lambda)^{\frac{\kappa-8}{2\kappa}} \, {}_2F_1\!\left(\frac{\kappa-8}{\kappa}, \frac{3(\kappa-8)}{2\kappa}; \frac{\kappa+8}{2\kappa}; 1-\lambda\right)$ have different non zero limits as $\lambda=1-e^{i0-}$ and $\lambda=1-e^{i0+}$, we must have $c_2=c_3=0$ in~\eqref{eq:solution-bulk}. Consequently, we have $\Delta(\lambda) \propto \lambda^{-\frac{\kappa-8}{\kappa}} (1-\lambda)^{\frac{\kappa-8}{2\kappa}}$, which implies~\eqref{eq:bulk-final}.
\end{proof}

For critical percolation ($\kappa=6$), the corresponding discrete convergence was shown by~\cite{conijn15}.
In~\cite{KSZ-connectivity}, the authors also expected that the factorization formula would hold for the critical Potts models. Our Theorem~\ref{thm:bulk} gives the rigorous extension to $\CLE_\kappa$ with general $\kappa\in(4,8)$.

One can also consider one-bulk and two-boundary correlations for other weights $\alpha$ at the interior point $z$, and solve~\eqref{eq:fusion-ode-bulk} at the ordinary point $\lambda_0=2$. Taking $\alpha=1-\frac{\kappa}{8}$ in~\eqref{eq:fusion-ode-bulk} corresponds to the one-point interior Green's function for bi-chordal $\SLE_\kappa$ pairs, while taking $\alpha=0$ corresponds to the probability that an interior point is between bi-chordal $\SLE_\kappa$ pairs. These cases were previously considered in~\cite{viklund-sle2} (for $\kappa\in(0,4]$), based on the construction of martingale observables by~\cite{KM13}.
However, for $\alpha\neq\frac{(3\kappa-8)(8-\kappa)}{32\kappa}$, the factorized solution~\eqref{eq:bulk-final} no longer exists. 

\appendix

\section{Boundary Green's functions of SLE}\label{appendix:gf}

In this appendix we show the equivalence of various definitions of SLE boundary Green's functions. For $\kappa\in(4,8)$, let $\eta$ be a $\SLE_\kappa$ on $\hH$ from $0$ to $\infty$.
The boundary Green's function in~\cite{fakhry2023} is defined to be the normalized limit of the probability that $\eta$ approaches the neighborhood of $x_1,...,x_n\in\R$. For $n=1,2$, the existence of such boundary Green's functions was earlier proved in~\cite{lawler-mink-R}.

\begin{proposition}[{\cite[Theorem 1.1]{fakhry2023}}]\label{prop:gf-existence}
For $x_1,...,x_n\in\R$, the limit
\begin{equation}\label{eq:gf-fz}
\wh G(x_1,...,x_n):=\lim_{r_i\to0}\prod_{i=1}^nr_i^{-h}\cdot\P[\eta\cap B(x_i,r_i)\neq\emptyset,1\le i\le n]
\end{equation}
exists and is continuous of $x_1,...,x_n$. Furthermore, the convergence is uniform on compact sets.
\end{proposition}

In this paper we use the $(1-h)$-dimensional Minkowski content $\nu_{\eta\cap\R}$ of $\eta\cap\R$ to define boundary Green's functions. For any open set $J\subset\R$, $\nu_{\eta\cap\R}(J)$ is defined by
\[
\nu_{\eta\cap\R}(J):=\lim_{\varepsilon\to0}\varepsilon^{-h}{\rm Leb}_\R(\{x\in J: {\rm \dist}(x,\eta\cap \R)<\varepsilon\});
\]
the limit is shown to exist by~\cite{lawler-mink-R} (see also~\cite{zhan-boundary-gf}). Then we have
\begin{proposition}\label{prop:equivalence}
    For any disjoint $S_1,..,S_n\subset\R$ and $\wh G(x_1,...,x_n)$ defined in~\eqref{eq:gf-fz}, there exists $C\in(0,\infty)$ such that
    \begin{equation*}
        \int_{S_1\times...\times S_n}\wh G(x_1,...,x_n)\prod_{i=1}^ndx_i=C^n\E\left[\prod_{i=1}^n\nu_{\eta\cap\R}(S_i)\right].
    \end{equation*}
    Here the expectation $\E$ is with respect to the chordal $\SLE_\kappa$ curve $\eta$.
    Hence, $\wh G(x_1,...,x_n)=CG(x_1,...,x_n)$, where $G(x_1,...,x_n):=G_{\hH,0,\infty}(x_1,...,x_n)$ is defined in~\eqref{eq:gf}. 
\end{proposition}

Proposition~\ref{prop:equivalence} is based on the following equivalence of defining Minkowski contents of $\eta\cap\R$ using the neighborhoods in $\R$ or in $\hH$, which is implicit in~\cite{lawler-mink-R,zhan-boundary-gf} and we provide a sketch argument here. 
\begin{lemma}\label{lem:mink}
The $(1-h)$-dimensional Minkowski content $\wh \nu_{\eta\cap\R}$ of $\eta\cap\R$ exists when viewed as a subset of $\ol\hH$, and $\wh \nu_{\eta\cap\R}=C \nu_{\eta\cap\R}$ for some $C\in(0,\infty)$. Namely, for any open subset $J\subset\R$, define
\[
\wh \nu_{\eta\cap\R}(J):=\lim_{\varepsilon\to0}\varepsilon^{-h}{\rm Area}(\{x\in \hH: {\rm \dist}(x,\eta\cap J)<\varepsilon\}).
\]
Then the limit exists and $\wh \nu_{\eta\cap\R}(J)=C\nu_{\eta\cap\R}(J)$.
\end{lemma}
\begin{proof}
Based on Proposition~\ref{prop:gf-existence}, repeating the proof of~\cite[Theorem 6.17]{zhan-boundary-gf} (with Eq.(6.11) there replaced by~\cite[Proposition 3.1]{fakhry2023}) gives the existence of such $\wh \nu_{\eta\cap\R}$. Note that $\wh \nu_{\eta\cap\R}$ and $\nu_{\eta\cap\R}$ are both supported on $\eta\cap\R$, and satisfy the conformal covariance and domain Markov property. According to the axiomatic characterization of $\nu_{\eta\cap\R}$~\cite{alberts-sheffield-bdy-measure} (see also~\cite{CL24}), $\nu_{\eta\cap\R}$ agrees with $\wh \nu_{\eta\cap\R}$ up to a multiplicative constant.
\end{proof}

\begin{proof}[Proof of Proposition~\ref{prop:equivalence}]
Repeating the derivation of~\cite[Eq.(6.15)]{zhan-boundary-gf} line by line, we have
\[
\int_{S_1\times...\times S_n}\wh G(x_1,...,x_n)\prod_{i=1}^ndx_i=\E\left[\prod_{i=1}\wh \nu_{\eta\cap\R}(S_i)\right].
\]
The result follows by combining with Lemma~\ref{lem:mink}. The constant $C$ is the same as in Lemma~\ref{lem:mink}.
\end{proof}

In~\cite{zhan-boundary-gf}, the author also showed the existence of the limits
\begin{equation}\label{eq:gf-R1}
\wt G(x):=\lim_{r\to0} r^{-h}\P[\eta\cap I(x,r)\neq\emptyset],\quad \wt G(x_1,x_2):=\lim_{r_1,r_2\to0} r_1^{-h}r_2^{-h}\P[\eta\cap I_{x_i}(r)\neq\emptyset,1\le i\le 2]
\end{equation}
where $I(x,r):=(x-r,x+r)\subset\R$. Furthermore, by~\cite[Theorem 6.17]{zhan-boundary-gf}, $\int_{S}\wt G(x)dx=\E\left[\nu_{\eta\cap\R}(S)\right]$ and $\int_{S_1\times S_2}\wt G(x_1,x_2)dx_1dx_2=\E\left[\prod_{i=1}^2\nu_{\eta\cap\R}(S_i)\right]$ for any open $S,S_1,S_2\subset\R$. Hence, $G(x)=\wt G(x)$ and $G(x_1,x_2)=\wt G(x_1,x_2)$ (recall~\eqref{eq:gf}). As a corollary of Proposition~\ref{prop:equivalence}, we have

\begin{corollary}\label{cor:equivalence}
Let $C$ be the same constant as in Proposition~\ref{prop:equivalence}. Then
we have $G(x)=\wt G(x)=C^{-1}\wh G(x)$ and $G(x_1,x_2)=\wt G(x_1,x_2)=C^{-2}\wh G(x_1,x_2)$.
\end{corollary}

We also need to deal with the $\SLE_\kappa$ bubble measure. Let $\mu^\bub_0$ be the $\SLE_\kappa$ bubble measure rooted at $0$~\eqref{eq:bubble-def}, and denote $\gamma$ as a sample of $\mu^\bub_0$.
\begin{proposition}\label{prop:gf-existence-bub}
For $x_1,...,x_n\in\R$, the limit
\[
\wh G^\bub_0(x_1,...,x_n)=\lim_{r_i\to0}\prod_{i=1}^nr_i^{-h}\cdot \mu^\bub_0[\gamma\cap B(x_i,r_i)\neq\emptyset,1\le i\le n]
\]
exists and is continuous of $x_1,...,x_n$. Furthermore, for any disjoint $S_1,..,S_n\subset\R$,
    \begin{equation}\label{eq:gf-mink-bub}
        \int_{S_1\times...\times S_n}\wh G^\bub_0(x_1,...,x_n)\prod_{i=1}^ndx_i=C^n\mu^\bub_0\left[\prod_{i=1}^n\nu_{\gamma\cap\R}(S_i)\right].
    \end{equation}
Hence, $\wh G^\bub_0(x_1,...,x_n)=C^n G^\bub_0(x_1,...,x_n)$ where $G^\bub_0(x_1,...,x_n)$ is defined in~\eqref{eq:def-gf-bub} and $C$ is the same constant as in Proposition~\ref{prop:equivalence}.
\end{proposition}
Note that by conformal covariance, we have $G^\bub_0(x)\propto H_{\hH}(0,x)^h$ where $H_{\hH}$ is the boundary Poisson kernel on $\hH$.
\begin{proof}[Proof of Proposition~\ref{prop:gf-existence-bub}]
Choose $\varepsilon$ small and fixed, and let $\tau_\varepsilon$ be the first hitting time of $\partial B(0,\varepsilon)$ for $\gamma$. Then conditioned on $\gamma[0,\tau_\varepsilon]$, the remaining part of $\gamma$ is a chordal $\SLE_\kappa$ from $\gamma(\tau_\varepsilon)$ to $0$ on the remaining domain $\cU(\gamma[0,\tau_\varepsilon])$. Thus the first statement follows from Proposition~\ref{prop:gf-existence} and conformal covariance of Green's function. Namely, $\wh G^\bub_0(x_1,...,x_n)=\mu^\bub_0[\wh G_{\cU(\gamma[0,\tau_\varepsilon]);\gamma(\tau_\varepsilon),0}(x_1,...,x_n)]$ while the derivatives of the conformal map $\psi_\varepsilon:\hH\setminus\gamma[0,\tau_\varepsilon]\to\hH$ on $x_1,...,x_n$ have uniform upper and lower bounds. Furthermore, this gives similar bounds of $G^\bub_0$ of~\cite[Proposition 3.1]{fakhry2023}, which ensures the validity of repeating the argument of deriving~\cite[Eq.(6.15)]{zhan-boundary-gf} to give~\eqref{eq:gf-mink-bub}. 
\end{proof}

We now show that the boundary Green's function $G_{\hH,y_1,y_2}(x_1,...,x_n)$ of chordal $\SLE_\kappa$ defined in~\eqref{eq:gf} converges to $G^\bub_0(x_1,...,x_n)$ as $y_1,y_2\to0$.
\begin{proposition}\label{prop:gf-convergence}
As $y_1,y_2\to0$, we have $|y_2-y_1|^{-h}G_{\hH,y_1,y_2}(x_1,...,x_n)\to G_0^\bub(x_1,...,x_n)$.
\end{proposition}
\begin{proof}
Let $\mu_{\hH,y_1,y_2}^\#$ be the probability measure of chordal $\SLE_\kappa$ from $y_1$ to $y_2$.
Then as $y_1,y_2\to0$, $|y_2-y_1|^{-h}\P_{y_1,y_2}$ converges weakly in Hausdorff topology to $\mu_0^\bub$ (see~\cite[Theorem 3.10]{zhan-bubble}). Denote $p(y_1,y_2;r_1,r_2):=|y_2-y_1|^{-h}\prod_{i=1}^n r_i^{-h}\P_{y_1,y_2}[\eta\cap B(x_i,r_i)\neq\emptyset,\ 1\le i\le n]$. Hence,
\[
\lim_{y_1,y_2\to0}p(y_1,y_2;r_1,r_2)=\prod_{i=1}^n r_i^{-h}\mu_0^\bub[\eta\cap B(x_i,r_i)\neq\emptyset,\ 1\le i\le n].
\]
Taking $r_i\to0$, by Proposition~\ref{prop:gf-existence-bub}, the right side becomes $\wh G_0^\bub(x_1,...,x_n)=C^nG_0^\bub(x_1,...,x_n)$ (with the same constant $C$ there). For the left side, note that $\lim_{r_1,r_2\to0}p(y_1,y_2;r_1,r_2)=C^n|y_2-y_1|^{-h}G_{\hH,y_1,y_2}(x_1,...,x_n)$ by Propositions~\ref{prop:gf-existence} and~\ref{prop:equivalence}. It suffices to show that the convergence of $p(y_1,y_2;r_1,r_2)$ as $r_i\to0$ is uniform of $y_1,y_2$, since we can then change the order of limit
\[
\lim_{y_1,y_2\to0}\lim_{r_1,r_2\to0}p(y_1,y_2;r_1,r_2)=\lim_{r_1,r_2\to0}\lim_{y_1,y_2\to0}p(y_1,y_2;r_1,r_2)=C^nG_0^\bub(x_1,...,x_n)
\]
which implies the result. Such uniform convergence can be seen by choosing again a small and fixed $\varepsilon$, and let $\tau_\varepsilon$ be the first hitting time of $\partial B(0,\varepsilon)$ for $\eta$. Conditioned on $\eta[0,\tau_\varepsilon]$, the remaining part of $\eta$ is a chordal $\SLE_\kappa$ on $\cU(\eta[0,\tau_\varepsilon])$ from $\eta(\tau_\varepsilon)$ to $y_2$. Note that the convergence in Proposition~\ref{prop:gf-existence} is uniform on compact sets, as well as we have uniform control of the uniformization map of $\phi:\cU(\eta[0,\tau_\varepsilon])\to\hH$ near each $x_i$. The desired uniform convergence then follows.
\end{proof}

Finally, we include the following equivalent description of CLE boundary four-point Green's functions, which will be used in Appendix~\ref{appendix:discrete}.
This first needs a variant of Lemma~\ref{lem:bubble-cle} and Corollary~\ref{cor:sharp}.

\begin{lemma}\label{lem:bubble-cle-variant}
Suppose $I\subset\R$ is an interval, and $\Gamma$ is a $\CLE_\kappa$ configuration on $\hH$. For $x\in\R$, Let $T_\varepsilon(x)$ (resp.~$\wh T_\varepsilon(x)$) be the event that there exists a loop $\ell$ intersecting $I$ and $B(x,\varepsilon)$ (resp.~$I(x,\varepsilon)$; in case there are two or more such loops, we e.g.~take the leftmost one on $I$). Then the law of ${\bf 1}_{T_\varepsilon(x)}\ell$ and ${\bf 1}_{\wh T_\varepsilon(x)}\ell$, times $\varepsilon^{-h}$, converges to $C\mu_x^\bub$ restricted to intersect $I$ for some constant $C\in(0,\infty)$.
\end{lemma}
\begin{proof}
We focus on $\wh T_\varepsilon(x)$, and the case for $T_\varepsilon(x)$ is similar (with using Lemma~\ref{lem:mink}).
Consider the CLE exploration process by discovering loops that intersect $B(0,\varepsilon)$, then conformally map the remaining unbounded component to $\hH$ (with $I$ fixed), until there exists a loop intersecting both $B(0,\varepsilon)$ and $I$.
By the same argument as~\cite[Section 4]{shef-werner-cle}, as $\varepsilon\to0$, $\P[\wh T_\varepsilon(x)]$ is $\varepsilon^{h+o(1)}$; furthermore, conditioned on $\wh T_\varepsilon(x)$, the conditional law of $\ell$ converges to $\mu_0^\bub$ conditioned to intersect $I$. On the other hand, on the event $\wh T_\varepsilon(x)$, let $Q_{x,\varepsilon}$ be the law of the leftmost loop on $I$ that intersects with $I(x,\varepsilon)$. Then for any fixed bounded interval $U\subset\R$ such that $U\cap I=\emptyset$, consider the measure $\mathbbm{N}_\varepsilon:=\P[\wh T_\varepsilon(x)]^{-1}{\bf 1}_{x\in U}Q_{x,\varepsilon}(d\gamma)dx$ in the place of $\mathbbm{M}_\varepsilon$ in the proof of Proposition~\ref{prop:touching-equals-bubble}. Then similar to that proof, for compactly supported and continuous $f$, $\mathbbm{N}_\varepsilon[f]$ on the one hand converges to $C\int_U f(x,\gamma){\bf 1}_{\gamma\cap I\neq\emptyset}\mu^\bub_x(d\gamma) dx$ for some $C\in(0,\infty)$, while by the existence of boundary Minkowski content in~\cite[Theorem 6.17]{zhan-boundary-gf}, $\mathbbm{N}_\varepsilon[f]$ on the other hand converges to $\left(\lim_{\varepsilon\to0}\frac{\varepsilon^h}{\P[\wh T_\varepsilon(x)]}\right)\cdot\E\left[\int_U\sum_{\ell\in \Gamma}f(x,\ell){\bf 1}_{\ell\cap I\neq\emptyset}\nu_{\ell\cap\R}(dx)\right]$. In particular, the limit $\lim_{\varepsilon\to0}\frac{\varepsilon^h}{\P[\wh T_\varepsilon(x)]}\in(0,\infty)$ exists. The result then follows.
\end{proof}
\begin{remark}\label{rmk:poisson}
The above proof gives
$\int_U \left(\lim_{\varepsilon\to0}\varepsilon^{-h}\P[\wh T_\varepsilon(x)]\right)dx=\int_U\E\left[\sum_{\ell\in \Gamma}{\bf 1}_{\ell\cap I\neq\emptyset}\nu_{\ell\cap\R}(dx)\right]$. If we further take $I=I(y,\delta)$ and let $\delta\to0$, then it yields
\[
\lim_{\delta\to0}\lim_{\varepsilon\to0}\delta^{-h}\varepsilon^{-h}\P[\text{there exists a loop intersecting }I(x,\varepsilon)\text{ and } I(y,\delta)]=H_{\hH}(x,y)^h.
\]
Recall that the multiplicative constant in $H_{\hH}(x,y)$ is chosen in Proposition~\ref{prop:cle-gf-to-bubble}.
\end{remark}

\begin{proposition}\label{lem:CLE-shrink}
Suppose $x_1<x_2<x_3<x_4$, and $\Gamma$ is a $\CLE_\kappa$ configuration on $\hH$. We have
\begin{itemize}
    \item Let $E$ be the event that there exists a loop $\ell\in\Gamma$ such that $\ell\cap I(x_1,r_1)\neq\emptyset$, $\ell\cap B(x_i,r_i)\neq\emptyset$ for $i=2,3$, and $\ell\cap I(x_4,r_4)\neq\emptyset$. Then $\prod_{i=1}^4 r_i^{-h}\P[E]$ converges to $CG^{(1234)}(x_1,x_2,x_3,x_4)$ as taking first $r_1\to0$, then $r_2,r_3\to0$, finally $r_4\to0$.
    \item Let $F$ be the event that there exist two loops $\ell,\ell'\in\Gamma$ such that $\ell$ intersects $I(x_1,r_1)$ and $B(x_2,r_2)$, while $\ell'$ intersects $I(x_4,r_4)$ and $B(x_3,r_3)$. Then $\prod_{i=1}^4 r_i^{-h}\P[F]$ converges to $C'G^{(12)(34)}(x_1,x_2,x_3,x_4)$ as taking first $r_1\to0$, then $r_2\to0$, then $r_3\to0$, finally $r_4\to0$.
\end{itemize}
Here $C$ and $C'$ are two constants in $(0,\infty)$.
\end{proposition}
\begin{proof}
Denote $C\in(0,\infty)$ to be some constant varying from line to line.
By Lemma~\ref{lem:bubble-cle-variant}, $\prod_{i=1}^4 r_i^{-h}\P[E]$ converges to $C\prod_{i=2}^4\mu_0^\bub[\eta\cap B(x_i,r_i)\neq\emptyset,\ i=2,3;\ \gamma\cap I(x_4,r_4)\neq\emptyset]$ as $r_1\to0$. Let $\mu_{x_1,x_2,x_3}^\bub$ be the $\SLE_\kappa$ bubble measure rooted at $x_1,x_2,x_3$, such that $\mu_{x_1,x_2,x_3}^\bub(d\gamma)dx_1dx_2dx_3=\prod_{i=1}^3\nu_{\gamma\cap\R}(dx_i)\mu^\bub(d\gamma)$.
Taking $r_2,r_3\to0$ and using Proposition~\ref{prop:gf-existence-bub}, the limit becomes $Cr_4^{-h}G_{x_1}^\bub(x_2,x_3)\mu_{x_1,x_2,x_3}^\bub[\gamma\cap I(x_4,r_4)\neq\emptyset]$. Note that for $\gamma$ sampled from $\mu_{x_1,x_2,x_3}^\bub$, let $\gamma:=\eta\cup\eta'$ be such that $\eta$ is the curve segment from $x_1$ to $x_3$ and hits $x_2$. Then given $\eta$, the conditional law of $\eta'$ is a chordal $\SLE_\kappa$ on $\cU(\eta)$ from $x_3$ to $x_1$. Thus, taking $r_4\to0$, the limit (denoted as $L_1$) is then $G_{x_1}^\bub(x_2,x_3)$ times the boundary one-point Green's function at $x_4$ of $\mu_{x_1,x_2,x_3}^\bub$, which exists due to~\eqref{eq:gf-R1}. To show $L_1\propto G^{(1234)}(x_1,x_2,x_3,x_4)$, note that $r_4^{-h}G_{x_1}^\bub(x_2,x_3)\mu_{x_1,x_2,x_3}^\bub[\gamma\cap B(x_4,r_4)\neq\emptyset]$ has a limit $L_2$ as $r_4\to0$ by Proposition~\ref{prop:gf-existence}, and $L_1\propto L_2$ by Corollary~\ref{cor:equivalence}. Furthermore, $L_2$ is also proportional to the limit of $\prod_{i=2}^4r_i^{-h}\mu_0^\bub[\eta\cap B(x_i,r_i)\neq\emptyset,\ 2\le i\le4]$ as $r_2,r_3,r_4\to0$. By Proposition~\ref{prop:gf-existence-bub}, $L_2=G_0^\bub(x_1,x_2,x_3)$. Combined with~\eqref{eq:bub-2}, we conclude the first assertion.

For the second assertion, let $G$ be the event that there exists a loop $\ell'$ intersecting $I(x_4,r_4)$ and $B(x_3,r_3)$. Restricted to $G$ and given this $\ell'$, due to Lemma~\ref{lem:bubble-cle-variant} and Proposition~\ref{prop:gf-existence-bub}, the conditional probability of $F$ satisfies $r_1^{-h}r_2^{-h}\P[F|G,\ell']\to CH_{\cU(\ell')}(x_1,x_2)^h$ as $r_1\to0$ and then $r_2\to0$ (recall the convention that $H_{\cU(\ell')}(x_1,x_2):=0$ when $x_1$ or $x_2$ is not in $\ol{\cU(\ell')}$). Since $H_{\cU(\ell')}(x_3,x_4)\le H_{\hH}(x_3,x_4)$, by dominated convergence, Lemma~\ref{lem:bubble-cle-variant}, Proposition~\ref{prop:gf-existence-bub} and~\eqref{eq:bubble-two-pinned}, we have
\[
r_3^{-h}r_4^{-h}\E\left[{\bf 1}_G H_{\cU(\ell')}(x_1,x_2)^h\right]\to C\int H_{\cU(\gamma)}(x_1,x_2)^h \mu_{x_3,x_4}^\bub(d\gamma)
\]
as $r_1\to0$ and then $r_2\to0$.  By~\eqref{eq:bub-3}, the result then follows.
\end{proof}

\section{Discrete convergence}\label{appendix:discrete}

In this section we complete the discrete parts in the proofs of Theorems~\ref{thm:percolation} and~\ref{thm:fk}, as mentioned in Section~\ref{sec:kappa=6}. We first focus on the Bernoulli percolation, and it suffices to show the following
\begin{proposition}\label{prop:discrete}
Let $x_1<x_2<x_3<x_4$. There exist constants $C,C'\in(0,\infty)$ such that
\begin{align}
\lim_{\delta\to0}\delta^{-\frac{4}{3}}\P^\delta[x_1^\delta\leftrightarrow x_2^\delta\leftrightarrow x_3^\delta\leftrightarrow x_4^\delta]=C G^{(1234)}(x_1,x_2,x_3,x_4),\label{eq:discrete-1234}\\
\lim_{\delta\to0}\delta^{-\frac{4}{3}}\P^\delta[x_1^\delta\leftrightarrow x_2^\delta\not\leftrightarrow x_3^\delta\leftrightarrow x_4^\delta]=C' G^{(12)(34)}(x_1,x_2,x_3,x_4).\label{eq:discrete-12-34}
\end{align}
\end{proposition}
\begin{proof}
The existence of these two limits was shown in~\cite{cf2025}.
For any disjoint $A,B\subset\hH$, we write $\{A\leftrightarrow B\}$ for the event that there exists an open path connecting $A$ and $B$.
Without loss of generality, we suppose $\min|x_i-x_j|>1$. For $s>0$, let $r_i<s$.
Define
\[
V:=\{I(x_1,r_1)\leftrightarrow B(x_2,r_2)\leftrightarrow I(x_3,r_3)\leftrightarrow B(x_4,r_4)\},
\]
and $W_i$ (resp.~$\wt W_i$) to be the event that there exists an open path connecting $I(x_i,r_i)$ (resp.~$B(x_i,r_i)$) and $B(x_i,1)$. Then by~\cite[Proposition 10]{conijn15}, there exists increasing functions $\varepsilon(s), m(s)$ of $s$ with $\varepsilon(s),m(s)\to0$ as $s\to0$, such that for all $\delta<m(s)$,
\[
\left|\frac{\P^\delta[x_1^\delta\leftrightarrow x_2^\delta\leftrightarrow x_3^\delta\leftrightarrow x_4^\delta]}{\prod_{i=1}^4\P^\delta[x_i^\delta\leftrightarrow B(x_i,1)]}\Bigg/\frac{\P^\delta[V]}{\P^\delta[W_1]\P^\delta[\wt W_2]\P^\delta[W_3]\P^\delta[\wt W_4]}-1\right|< \varepsilon(s).
\]
By~\cite[Theorem 1.1]{DGLZ24}, for each $1\le i\le 4$, $\P^\delta[x_i^\delta\leftrightarrow B(x_i,1)]=C_1\delta^{\frac{1}{3}}(1+o(1))$ for some constant $C_1\in(0,\infty)$. By~\cite[Theorem 1.9]{cf2025}, $\delta^{-\frac{4}{3}}\P^\delta[x_1^\delta\leftrightarrow x_2^\delta\leftrightarrow x_3^\delta\leftrightarrow x_4^\delta]$ converges to a limit $L\in(0,\infty)$ as $\delta\to0$.
On the other hand, when we first let $\delta\to0$, due to the full scaling limit convergence of critical Bernoulli percolation to $\CLE_6$~\cite{camia-newman-06}, $\P^\delta[V]\to\P[E]$ as $\delta\to0$, where the event $E$ is defined in Proposition~\ref{lem:CLE-shrink}. Meanwhile, $\P^\delta[W_i]$ and $\P^\delta[\wt W_i]$ converge to their $\CLE_6$ analogs, denoted by $\P[W_i]$ and $\P[\wt W_i]$, as $\delta\to0$. Therefore, for all $r_i< s$, we have
\begin{equation}\label{eq:discrete-1}
L(1+\varepsilon(s))^{-1}\le C_1^4\left(\prod_{i=1}^4 r_i^{-\frac{1}{3}}\P[W_1]\P[\wt W_2]\P[W_3]\P[\wt W_4]\right)^{-1}\left(\prod_{i=1}^4 r_i^{-\frac{1}{3}}\P[E]\right)\le L(1-\varepsilon(s))^{-1}.
\end{equation}
Now we take the limit of~\eqref{eq:discrete-1} in the following order $\mathcal{O}$: first $r_1\to0$, then $r_2,r_4\to0$, finally $r_3\to0$. By Proposition~\ref{lem:CLE-shrink}, under the limit order $\mathcal{O}$, we have $\prod_{i=1}^4 r_i^{-\frac{1}{3}}\P[E]\to C_2 G^{(1234)}(x_1,x_2,x_3,x_4)$ for some $C_2\in(0,\infty)$. If we write $\ol C$ and $\underline{C}$ be the limsup and liminf of $\prod_{i=1}^4 r_i^{-\frac{1}{3}}\P[W_1]\P[\wt W_2]\P[W_3]\P[\wt W_4]$ under the limit order $\mathcal{O}$, then~\eqref{eq:discrete-1} yields
\[
L(1+\varepsilon(s))^{-1}\le \ol C^{-1}C_1^4C_2G^{(1234)}(x_1,x_2,x_3,x_4)\le \underline{C}^{-1} C_1^4C_2G^{(1234)}(x_1,x_2,x_3,x_4)\le L(1-\varepsilon(s))^{-1}.
\]
Taking $s\to0$ yields $\underline{C}=\ol C$, and hence $L=CG^{(1234)}(x_1,x_2,x_3,x_4)$ for some $C\in(0,\infty)$. This gives~\eqref{eq:discrete-1234}.

Now we consider $\P^\delta[x_1^\delta\leftrightarrow x_2^\delta\not\leftrightarrow x_3^\delta\leftrightarrow x_4^\delta]$. Though it is not an increasing event (hence we may not use~\cite[Proposition 10]{conijn15} directly), we can instead consider $\P^\delta[x_1^\delta\leftrightarrow x_2^\delta, x_3^\delta\leftrightarrow x_4^\delta]$, and define
\[
V':=\{I(x_1,r_1)\leftrightarrow B(x_2,r_2),\ I(x_3,r_3)\leftrightarrow B(x_4,r_4)\}.
\]
Then one can similarly use Proposition~\ref{lem:CLE-shrink} to show that
\[
\lim_{\delta\to0}\delta^{-\frac{4}{3}}\P^\delta[x_1^\delta\leftrightarrow x_2^\delta, x_3^\delta\leftrightarrow x_4^\delta]=C'G^{(12)(34)}(x_1,x_2,x_3,x_4)+C''G^{(1234)}(x_1,x_2,x_3,x_4)
\]
for some constant $C',C''\in(0,\infty)$.
Since $\lim_{\delta\to0}\frac{\P^\delta[x_1^\delta\leftrightarrow x_2^\delta\leftrightarrow x_3^\delta\leftrightarrow x_4^\delta]}{\P^\delta[x_1^\delta\leftrightarrow x_2^\delta, x_3^\delta\leftrightarrow x_4^\delta]}$ tends to $1$ as $x_2\to x_3$, we find $C''=C$. Combined with~\eqref{eq:discrete-1234}, we obtain~\eqref{eq:discrete-12-34}.
\end{proof}

\begin{remark}
    As a by-product, the above proof indeed shows that $\lim_{\delta\to0}\P^\delta[W_1]=Cr_1^{\frac{1}{3}}(1+o(1))$ and $\lim_{\delta\to0}\P^\delta[\wh W_1]=\wh Cr_1^{\frac{1}{3}}(1+o(1))$ as $r_1\to0$ for some constants $C,\wh C\in(0,\infty)$.
\end{remark}

Note that the proof of~\cite[Proposition 10]{conijn15} is based on RSW estimates and FKG inequality to derive a variant of the one-arm event coupling in~\cite{gps-pivotal}. Both of these have counterparts on the critical FK-Ising model, see~\cite{FK-FKG,FK-RSW}. The full scaling limit $\CLE_{16/3}$ of the critical FK-Ising model was established in~\cite{kemppainen2019}. The sharp asymptotics $\P^\delta_{\rm FK}[x_i^\delta\leftrightarrow B(x_i,1)]=C\delta^{\frac{1}{2}}(1+o(1))$ for the FK-Ising model was derived in~\cite[Eq.(1.7)]{CF-FK}. Thus, we have all the needed ingredients for the FK-Ising model, and the FK-Ising analog of Proposition~\ref{prop:discrete} follows similarly.

\section{MATLAB code}\label{appendix:matlab}

Here we provide MATLAB code that verifies the derivation of the third-order ODE~\eqref{eq:ode} obtained by substituting~\eqref{eq:g0-U} into~\eqref{eq:fusion-pde}.

\begin{footnotesize}
\begin{lstlisting}
clear
syms u x1 x2 x3 kappa real
h = 8/kappa - 1;

% Difference variables
delta_1_u = x1 - u;      % x1 - u
delta_2_u = x2 - u;      % x2 - u  
delta_3_u = x3 - u;      % x3 - u

delta_2_1 = x2 - x1;     % x2 - x1
delta_3_1 = x3 - x1;     % x3 - x1
delta_3_2 = x3 - x2;     % x3 - x2

delta_u_3 = u - x3;      % u - x3

% Cross-ratio
lambda = (delta_1_u * delta_3_2) / (delta_3_1 * delta_2_u);

% Symbolic derivatives of U, corresponds to U, U', U'', U'''
syms U0 U1 U2 U3

% Prefactor
prefactor = ((delta_3_1 * delta_2_u) / (delta_1_u * delta_3_2 * delta_2_1 * delta_u_3))^(2*h);

%% Derivatives of lambda, prefactor with respect to u
lambda_u  = diff(lambda, u);
lambda_uu = diff(lambda, u, 2);
lambda_uuu = diff(lambda, u, 3);

prefactor_u   = diff(prefactor, u);
prefactor_uu  = diff(prefactor, u, 2);
prefactor_uuu = diff(prefactor, u, 3);

%% Chain rule expansion
U_u   = U1 * lambda_u;
U_uu  = U2 * lambda_u^2 + U1 * lambda_uu;
U_uuu = U3 * lambda_u^3 + 3 * U2 * lambda_u * lambda_uu + U1 * lambda_uuu;

% g function and its derivatives
g_u = prefactor_u * U0 + prefactor * U_u;

g_uuu = prefactor_uuu * U0 ...
      + 3 * prefactor_uu * U_u ...
      + 3 * prefactor_u * U_uu ...
      + prefactor * U_uuu;

%% Summation of derivatives with respect to x1, x2, x3
sum_terms1 = 0; 
sum_terms2 = 0;

for xi = [x1 x2 x3]
    % Derivatives of lambda, U, g with respect to xi
    lambda_xi  = diff(lambda, xi);
    lambda_u_xi = diff(lambda_u, xi);  % Derivative of lambda_u with respect to xi

    U_xi = U1 * lambda_xi;
    U_u_xi = U2 * lambda_u * lambda_xi + U1 * lambda_u_xi;  % Mixed partial derivative

    g_xi = diff(prefactor, xi) * U0 + prefactor * U_xi;

    % Mixed partial derivative of g_u with respect to xi
    g_u_xi = diff(prefactor_u, xi) * U0 ...
           + diff(prefactor, xi) * U_u ...
           + prefactor_u * U_xi ...
           + prefactor * U_u_xi;

    % Accumulate terms
    sum_terms1 = sum_terms1 + (4 * g_xi / (xi - u)^2 - 8 * h * (prefactor * U0) / (xi - u)^3);
    sum_terms2 = sum_terms2 + (4 * g_u_xi / (xi - u) - 4 * h * g_u / (xi - u)^2);
end

coefficient = (1 - 8/kappa) / 2;

% Let v be the variable representing the cross-ratio
syms v

% Result of substitution
expression = kappa/4 * g_uuu + coefficient * sum_terms1 + sum_terms2;

%% Case 1: u = -1, x1 = 0, x2 = 1, x3 expressed in terms of v
expr_case1 = subs(expression, {u, x1, x2}, {-1, 0, 1});
expr_case1 = subs(expr_case1, {x3}, {1/(1-2*v)});

%% Case 2: u = 0, x2 = 1, x3 = 2, x1 expressed in terms of v  
expr_case2 = subs(expression, {u, x2, x3}, {0, 1, 2});
expr_case2 = subs(expr_case2, {x1}, {(2*v)/(1+v)});

% Simplify and output
simplify(expr_case1)
simplify(expr_case2)
\end{lstlisting}
\end{footnotesize}

Both of the two outputs, after combining like terms, give the same third-order ODE as~\eqref{eq:ode}.

\bibliographystyle{alpha}
\footnotesize{
\newcommand{\etalchar}[1]{$^{#1}$}
\def\cprime{$'$}

}

\end{document}